\documentclass[12pt, reqno]{amsart}
\usepackage{graphicx}
\usepackage[hidelinks]{hyperref}
\usepackage{mathrsfs}
\usepackage{amsfonts,amsmath,amsthm, amssymb}
\usepackage{latexsym, euscript, epic, eepic}
\usepackage{indentfirst}
\usepackage{tikz}
\usepackage{enumerate}

\usepackage{fancyhdr}
\fancypagestyle{plain}{%
  \fancyhf{}%
  \fancyfoot[C]{\thepage}%
}

\usepackage[centering]{geometry}
\usepackage[square, comma, sort&compress, numbers]{natbib}
\theoremstyle{plain}
\newtheorem{theorem}{Theorem}[section]
\newtheorem{remark}{Remark}[section]

\newtheorem{lemma}[theorem]{Lemma}
\newtheorem{corollary}[theorem]{Corollary}
\newtheorem{proposition}[theorem]{Proposition}
\newtheorem{claim}[theorem]{Claim}

\theoremstyle{definition}
\newtheorem{definition}{Definition}[section]

\numberwithin{equation}{section}
\newtheoremstyle{citing}
{3pt}
{3pt}
{\itshape}
{}
{\bfseries}
{.}
{.5em}
{\thmnote{#3}}
\theoremstyle{citing}

\begin{document}
\title{ Intersections of Quantitative Recurrence and Badly Approximable Sets for Expanding Markov Maps
}
\author{Lingmin Liao$^{1}$}
\address{$^{1}$School of Mathematics and Statistics, Wuhan University, Wuhan, Hubei, China}
\email{lmliao@whu.edu.cn}

\author{Na Yuan$^{2}$}
\address{$^{2}$School of Mathematics,
Guangdong University of Education,  Xingang Middle Road, 351, Guangzhou, P.~R.\ China}
\email{yuanna199@gmail.com}



\date{}
\maketitle

\begin{center}
\begin{minipage}{120mm}
{\small {\bf Abstract.}
We study a class of 
 expanding Markov maps on  the unit interval, which
 includes
sofic \(\beta\)-transformations, Gauss maps, tent maps and cookie-cutter maps.
We investigate the quantitative recurrence
property
for points
that are dynamically badly approximable, namely, points
whose orbits eventually remain a
positive distance from a given point. 
We show that the Hausdorff dimension of the intersection of a recurrence set and a badly approximable set is determined by the same pressure function as that of the corresponding recurrence set.
We further obtain the Hausdorff dimensions of two
related intersections: recurrence sets with non-recurrent sets,
and shrinking target sets with badly approximable sets. }
\end{minipage}
\end{center}

\section{Introduction}
\subsection{Background and motivation}
Let $(X,T)$ be a  one-dimensional
expanding dynamical system. 
We are interested
in points $x\in X$ whose orbits
\[
x, Tx, T^2x,\ldots
\]
satisfy two distinct asymptotic conditions.
The first is recurrence: the
orbit returns close to its initial point at a
prescribed rate. 
The second is avoidance: the orbit  eventually remains a positive distance from a given point.

The guiding question is whether these two conditions can hold simultaneously
without reducing the Hausdorff dimension of the recurrence set. 
More precisely, we ask the following questions:

\emph{Can an orbit return close to its initial point at a prescribed
quantitative rate while simultaneously avoiding a given point? If so, does
the avoidance requirement affect the Hausdorff dimension of the set of  recurring points?}

These questions are motivated by the connections between
Diophantine approximation, dynamical systems, and fractal geometry.

One of the classical
starting points of Diophantine approximation on fractal sets is Mahler's
question \cite{Mahler1984} concerning approximation properties of points in
the middle-third Cantor set.  
This problem has led to extensive work on the
distribution of rational points inside or near fractal sets; see
\cite{Bugeaud2016,Levesley2007,Khalil36,Heweikun2024,Heweikun2025,Tanwang2024}
and references therein.

Two fundamental families of sets in classical Diophantine approximation are
well-approximable sets and badly approximable sets. Given a function
\(\psi:\mathbb N\to(0,\infty)\), the set of \(\psi\)-well approximable points is
defined by
\[
W(\psi)
=
\left\{
x\in [0,1]:
\left|x-\frac pq\right|<\psi(q)
\ \text{for infinitely many }(p,q)\in\mathbb Z\times\mathbb N
\right\}.
\]
The metric theory of \(W(\psi)\) has been extensively studied; see, for
instance,
\cite{Khintchine1924,Jarnik1931,Beresnevich2006}.
In contrast, the set of badly approximable points is defined by
\[
\mathrm{BAD}
=
\left\{
x\in[0,1]:
\exists c(x)>0 \text{ such that }
\left|x-\frac pq\right|>\frac{c(x)}{q^2}
\text{ for all }(p,q)\in\mathbb Z\times\mathbb N
\right\}.
\]
It is well known that \(\mathrm{BAD}\) has Lebesgue measure zero but full
Hausdorff dimension \cite{Schmidt1980}.
Badly approximable points on
fractal sets have also been studied extensively; 
we refer the reader to \cite{Kleinbock2005,Kleinbock2010,Das2018,Das2019}
for related results.

Many problems in Diophantine approximation have natural dynamical formulations.
Let $(X,\rho)$ be a metric space and  $T:X\to X$ be a transformation.
 We study the 
 approximation properties of the orbit \( \{T^n x\}_{n\geq1} \) of a point \(x\in X\).
Two important dynamical analogues of well-approximable sets are shrinking target sets and quantitative recurrence sets.

Given a  point $y\in X$ and a function
$\psi:\mathbb N\times X\to (0,\infty)$, define
\[
S(T,\psi,y)
:=
\left\{
x\in X:
\rho(T^n x,y)<\psi(n,x)
\ \text{for infinitely many } n\in\mathbb N
\right\}.
\]
The shrinking target problem, which  concerns  the size of $S(T,\psi,y)$ in terms of
measure and dimension,  was introduced and investigated by Hill and Velani
\cite{Hill1995,Hill1997}.
Since then, shrinking target sets have been studied in a variety of dynamical and fractal settings, including expanding maps, Gauss maps, $\beta$-transformations, conformal iterated function systems, matrix actions, and Bedford--McMullen carpets; we refer the reader to \cite{Urbanski,Li2014,Shen2013,li2022,barany2018}.

Closely related to shrinking target sets are quantitative recurrence sets. 
Given a function $\tilde\psi:\mathbb N\times X\to (0,\infty)$, we are interested in the following  set
\[
R\langle T,\tilde\psi\rangle
=
\left\{
x\in X:
\rho(T^n x,x)<\tilde\psi(n,x)
\ \text{for infinitely many } n\in\mathbb N
\right\}.
\]
Thus, \(R\langle T,\tilde\psi\rangle\) consists of points whose orbits return infinitely often to neighbourhoods of their initial positions with radii prescribed by $\tilde\psi$.

The metric theory of \(R\langle T,\tilde\psi\rangle\) has been developed 
from both
measure-theoretic and dimensional viewpoints. 
Zero--full measure laws for \(R\langle T,\tilde\psi\rangle\)
have been obtained under various dynamical assumptions; see, for example,
\cite{Chang2018,Baker2019,Hussain2021}. 
Hausdorff dimension results were
proved for $\beta$-transformations by  \cite{Tan2011}, and for
conformal iterated function systems (including the continued fraction system)  by  \cite{Seuret2013}.
Further developments can be found in
\cite{He2023QuantitativeRP,persson2024,koivusalo2018,Yuan2023}.

As an analogue of the badly approximable set,
for a given point $y\in X$, define
\[
\mathrm{Bad}(T,y)
=
\left\{
x\in X:
\liminf_{n\to\infty}\rho(T^n x,y)>0
\right\}.
\]
Thus, $\mathrm{Bad}(T,y)$ consists of points whose orbits eventually remain
a positive distance from $y$.
For the Gauss map $T_G$, the set $\mathrm{Bad}(T_G,0)$ coincides with the classical set $\mathrm{BAD}$ and is known to be winning in the sense of Schmidt’s game \cite{Schmidt1966}, hence has full Hausdorff dimension.
Winning properties of the set $\mathrm{Bad}(T,y)$ have been proved for various classes of expanding maps, including $\beta$-transformations and the Gauss map, as well as for toral endomorphisms; see \cite{Tseng2009,Hu2014,Broderick2011,Hu2025}.

We also consider the non-recurrent set
\[
N(T)
=
\left\{
x\in X:
\liminf_{n\to\infty}\rho(T^n x,x)>0
\right\}.
\]
The points in $N(T)$ have orbits that eventually
remain a positive distance from their
initial positions. 
The Hausdorff dimension and winning properties of such sets have
 been investigated in expanding and product dynamical systems
\cite{yuan2022,huang2025}.

Although shrinking target sets, quantitative recurrence sets, badly
approximable sets and non-recurrent sets have been studied separately, less
is known about the intersections that combine ``(self)" approximation and avoidance
conditions.

A closely related intersection problem has recently been studied in classical
Diophantine approximation. For $\eta\in\mathbb R\setminus\mathbb Z$,
let
\[
\mathrm{Bad}^{\eta}
=
\left\{
x\in[0,1]:
\inf_{q\geq1} q\|qx-\eta\|>0
\right\},
\]
where $\|\cdot\|$ denotes the distance to the nearest
integer.
With the above definition of $W(\psi)$, Hussain and Simmons
\cite{HussainSimmons2026}
proved that, for every $\tau\geq1$ and
$\psi_\tau(q)=q^{-(\tau+1)}$,
\[
\dim_{\mathrm H}
\bigl(W(\psi_\tau)\cap\mathrm{Bad}^{\eta}\bigr)
=
\dim_{\mathrm H}W(\psi_\tau)
=
\frac{2}{\tau+1}.
\]
Thus, for every $\tau\geq1$, the Hausdorff dimension of the
$\psi_\tau$-well-approximable set $W(\psi_\tau)$ is unchanged upon
intersection with the inhomogeneous badly approximable set
$\mathrm{Bad}^{\eta}$.

In several dynamical settings, quantitative recurrence sets satisfy a
large intersection property introduced by Falconer \cite{Falconer1990}.
For $s\in(0,m]$, 
the class $\mathcal G^s(\mathbb R^m)$ consists of $G_\delta$
sets with a large intersection property.
 In particular, it is closed under
countable intersections, and every $F\in\mathcal G^s(\mathbb R^m)$ satisfies
\(\dim_{\mathrm H}F\geq s
\), where $\dim_{\mathrm H}$ denotes the Hausdorff dimension.
We write $\mathcal G^s(\mathbb T^m)$ for the corresponding
large intersection class on the torus.
 Hu and Persson \cite{persson2024} proved that, for hyperbolic
toral automorphisms of \(\mathbb T^2\) and the function $\tilde\psi(n,x)=e^{-n\alpha}$, there is an  $s_{\tilde\psi}>0$ such that
\[
    R\langle T,\tilde\psi\rangle\in\mathcal G^{s_{\tilde\psi}}(\mathbb T^2) \quad
   \text{and} \quad
    \dim_{\mathrm H}R\langle T,\tilde\psi\rangle=s_{\tilde\psi}.
\]
For radii of the form
\(\tilde\psi(n,x)=\exp(-S_n f(x))\), where $f$ is positive and
continuous,
analogous large intersection results for
$\beta$-transformations and the Gauss map follow from
\cite{He2025MTP}.

The set $\mathrm{Bad}(T,y)$ behaves differently.
Existing results establish winning
properties in the sense of Schmidt's game or related variants; see
\cite{Tseng2009,Hu2014,Broderick2011,Hu2025}. These results do not imply that
$\mathrm{Bad}(T,y)$ belongs to a Falconer large intersection class.
The
difference can already be seen for a $\beta$-transformation. 
For $\beta>1$,
let
\[
    T_\beta(x)=\beta x-\lfloor\beta x\rfloor,\qquad x\in[0,1).
\]
 Fix \(\tau>0\) and define
\[
W_{\tau}(T_{\beta}, y)
:=
\left\{
x\in[0,1]:
\lvert T_\beta^n x-y\rvert<e^{-\tau n}
\ \text{for infinitely many }n\in\mathbb{N}
\right\}.
\]
By \cite{He2025MTP}, 
\(W_{\tau}(T_{\beta}, y)\in\mathcal{G}^{s_{\tau}}([0,1])
\)
for some \(s_{\tau}>0\). 
By the definition of
\(\mathcal G^{s_{\tau}}([0,1])\), the set \(W_{\tau}(T_{\beta}, y)\) is a dense
\(G_\delta\) subset of \([0,1]\). Hence, \(W_{\tau}(T_{\beta}, y)\) is residual.
On the other hand,
every
$x\in W_{\tau}(T_{\beta}, y)$ satisfies
\[
\liminf_{n\to\infty}\lvert T_\beta^n x-y\rvert=0,
\]
so
\[
W_{\tau}(T_{\beta}, y)
\subset
[0,1]\setminus\mathrm{Bad}(T_\beta,y).
\]
Since the complement of the residual set $W_{\tau}(T_{\beta}, y)$ is meagre,
$\mathrm{Bad}(T_\beta,y)$ is meagre. 
 Since
\([0,1]\) is a Baire space and every member of
\(\mathcal{G}^{s}([0,1])\) is residual, we conclude that
\[
\mathrm{Bad}(T_\beta,y)
\notin\mathcal{G}^{s}([0,1])
\quad
\text{for every} \quad 0<s\leq1.
\]
Consequently, even when \(R(T,\psi)\) belongs to a Falconer
large intersection class, the size of
\[
R(T,\psi)\cap\mathrm{Bad}(T,y)
\]
cannot be determined solely from the countable-intersection property
of Falconer classes.

The main objective of this paper is to study the intersection \(R(T,\psi)\cap\mathrm{Bad}(T,y)\). 
The points in
\(R(T,\psi)\cap\mathrm{Bad}(T,y)\) satisfy two qualitatively different
orbit conditions: their orbits return close to their initial positions
at a prescribed rate, while eventually remaining uniformly separated
from a given target \(y\).
Under the hypotheses of our main theorems, the avoidance
condition does not reduce the Hausdorff dimension. More precisely,
\[
\dim_{\mathrm H}
\bigl(R(T,\psi)\cap\mathrm{Bad}(T,y)\bigr)
=
\dim_{\mathrm H}R(T,\psi).
\]

\subsection{Setting and main results}
Now, we give our  dynamical setting and state the main results. 
Throughout
the paper, $T$ is an expanding Markov map.
This class of  maps provides
a convenient framework in which quantitative
recurrence and avoidance properties can be treated simultaneously. 
It includes,
among other examples, sofic \(\beta\)-transformations, Gauss maps, tent maps and
cookie-cutter systems.
\begin{definition}\label{1.1}
		Let $X=\bigcup\limits_{i\in \Lambda}\overline{I(i)}\subset [0,1]$, 
        where  $\Lambda=\{1,2,\dots,d\}$  is a  finite or  countable ($d=+\infty$) index set with  $\sharp\Lambda\geq2$, $\{I(j)\}_{j\in\Lambda}$  are disjoint open intervals in $[0,1]$, and   $\overline{I(i)}$ is the closure of $I(i)$.
		A map $T:X \rightarrow[0,1]$ is called  an expanding  Markov map (EMM) if $T$ satisfies the following conditions.
		\begin{enumerate}[(i)]
			\item (\textbf{uniform ~lower ~ bound ~for ~derivative}) There exists $\alpha>0$ such that for any  $x\in \bigcup\limits_{i\in \Lambda}I(i)$, the derivative of $T$ at $x$ exists and satisfies $ |T^\prime(x)|\geq \alpha$.
			\item (\textbf{expanding})
            There exist \(\lambda>1\) and \(K\geq1\) such that for all \(x\in X\),
           if for all \(0\leq j\leq K-1\), \(T^j(x)\in\bigcup\limits_{i\in \Lambda}I(i)\), then
			\begin{equation*}
				\left|(T^K)^\prime(x)\right|  \geq \lambda.
			\end{equation*}
			\item (\textbf{weak} $ C^{1+\delta} $) There exist $0<\delta<1$ and $C>0$ such that for any  $i\in\Lambda$ and  any $x,y\in I(i)$,
			\begin{equation}\label{1.5}
				\left| \frac{T^\prime(x)}{T^\prime(y)}-1 \right|  \leq C|x-y|^{\delta}.
			\end{equation}
			\item (\textbf{mixing})   There exist $m_0\geq 0$ and  $j_i\in \Lambda$   $(1\leq i\leq m_0)$ such that for any $j,k\in \Lambda$,
\begin{equation}\label{1.6}
    I(j)\cap T^{-1} I(j_1)\cap \cdots \cap	T^{-m_0}(I(j_{m_0}))\cap T^{-m_0-1}(I(k))\neq \emptyset.
\end{equation}
			\item (\textbf{Markov ~property})  For any $i$, $j\in\Lambda$, if $I(i)\cap T\big(I(j)\big)\neq\emptyset$,  then $I(i)\subset T\big(I(j)\big)$.
		\end{enumerate}
	\end{definition}
In the countable alphabet case, we  assume the following stronger
version of the Markov property:

(v') (\textbf{full image property})
\(T(I(j))\supset(0,1)\) for every  \(j\in\Lambda\).

\begin{remark}
Our definition is based on the definition of expanding Markov maps in Ma\~{n}\'e's book \cite[p. 167]{mane1990}. 
However, we do not assume
\( 
\mathcal{L}\left(\bigcup_{i\in\Lambda}\overline{I(i)}\right)=1,
\)
where $\mathcal{L}$ denotes the Lebesgue measure. 
This allows us to include systems whose invariant sets are fractal, such as cookie-cutter systems.
\end{remark}
The \textit{repeller} of an EMM is defined by
\[
E
=
\bigcap_{n=1}^{\infty}
\bigcup_{(i_1,\ldots,i_n)\in\Lambda^n}
I(i_1)\cap T^{-1}I(i_2)\cap\cdots\cap T^{-(n-1)}I(i_n).
\]
 We remark that
all  results in this paper are stated for subsets of the repeller $E$.

Let \(f:[0,1]\to(0,\infty)\) be a positive continuous function. 
 For any $x\in E$ denote
\[
S_nf(x)
=
\sum_{k=0}^{n-1}f(T^k x).
\]
Define the quantitative recurrence set
\[
R[T,f]
=
\left\{
x\in E:
|T^n x-x|<e^{-S_nf(x)}
\text{ for infinitely many } n\in\mathbb N
\right\}.
\]
For a  function \(\psi:\mathbb N\to(0,\infty)\), we also define
\[
R(T,\psi)
=
\left\{
x\in E:
|T^n x-x|<\psi(n)
\text{ for infinitely many } n\in\mathbb N
\right\}.
\]
The dynamical badly approximable set with respect to the target \(y\in  E\) is
defined by
\[
\mathrm{Bad}(T,y)
=
\left\{
x\in E:
\liminf_{n\to\infty}|T^n x-y|>0
\right\}.
\]

The main object of the paper is the intersection of quantitative recurrence
sets with badly approximable sets. More precisely, we set
\[
R_{\mathrm{Bad}}[T,f,y]
:=
R[T,f]\cap \mathrm{Bad}(T,y)\]
and
\[R_{\mathrm{Bad}}(T,\psi,y)
:=
R(T,\psi)\cap \mathrm{Bad}(T,y).
\]

In the statements below, $P(T,\phi)$ denotes the topological pressure of $T$
for a function $\phi$; the definition is given in
Section~\ref{pressure-function}.

Our first main theorem gives the Hausdorff dimension of \(R_{\mathrm{Bad}}[T,f,y]\).
\begin{theorem}\label{theorem-f-finite}
Let $T$ be an EMM on $X=\bigcup_{i\in \Lambda}\overline{I(i)}$ with finite index set $\Lambda$, and let $f$ be a positive continuous function defined on $[0,1]$. 
Let  \(y\in E\).
Then
\[
\dim_{\mathrm{H}}R_{\mathrm{Bad}}[T,f,y]=\dim_{\mathrm{H}}R[T,f]=s^{*},
\]
where $s^{*}$ is the unique real root of $P(T, -t(f+\log|T'|))=0$.
\end{theorem}
The next theorem provides the corresponding result for countable alphabets.
\begin{theorem}\label{theorem-f-countable}
Let $T$ be an EMM on $X=\bigcup_{i\in \Lambda}\overline{I(i)}$ with countable index set $\Lambda$, and let $f$ be a positive continuous function defined on $[0,1]$.
Let  \(y\in E\).
Then
\[
 \dim_{\mathrm{H}}R_{\mathrm{Bad}}[T,f,y]=\dim_{\mathrm{H}}R[T, f]= \inf\{t \geq 0 : P(T, -t(f+\log|T'|))\leq 0\}.
\]
\end{theorem}

Since every orbit considered below lies in $E\subset(0,1)$, we may replace
$\psi$ by $\min\{\psi,1\}$.
We use this normalization throughout and
continue to write $\psi$. In particular, $0<\psi(n)\leq1$.

To state the results on $R_{\mathrm{Bad}}(T,\psi,y)$, we define $\gamma:=\gamma(\psi)\in [1,\infty]$ by
\[
\log\gamma
=
\liminf_{n\to\infty}
\frac{-\log\psi(n)}{n}.
\]
\begin{theorem}\label{main-theorem-finite}
    Let $T$ be an EMM on $X=\bigcup_{i\in \Lambda}\overline{I(i)}$ with finite $\Lambda$, and let $\psi$ be a positive function defined on $\mathbb{N}$. Let $y\in E$.
    \begin{enumerate}[$(1)$]
        \item If $1\leq \gamma<\infty$, then \[ \dim_{\mathrm{H}} R_{\mathrm{Bad}}(T,\psi,y) = \dim_{\mathrm{H}}R(T,\psi) = t^*, \]
        where $t^*$ is the unique real root of $P(T, -t(\log \gamma+\log|T'|))= 0$.
    
        \item If $\gamma=\infty$, then $\dim_{\mathrm{H}} R_{\mathrm{Bad}}(T,\psi,y) = \dim_{\mathrm{H}}R(T,\psi) = 0$.  
    \end{enumerate}
\end{theorem}
For countable alphabets, we prove the following result for
$1\leq\gamma<\infty$.
\begin{theorem}\label{theorem-countable}
		Let $T$ be an EMM on $X=\bigcup_{i\in \Lambda}\overline{I(i)}$ with countable $\Lambda$.  
		Let $\psi$ be a positive function. 
        Let $y\in E$.
        Suppose that $1\leq \gamma<\infty$.
        Then  
		\[
			\dim_{\mathrm{H}} R_{\mathrm{Bad}}(T,\psi,y) = \dim_{\mathrm{H}}R(T,\psi) = \inf\big \{t\geq0: P(T, -t(\log \gamma+\log|T'|))\leq 0\big \}.  
		\]
\end{theorem}
\begin{remark}
   We do not treat the case $\gamma=\infty$ for countable alphabets.
   In the
finite-alphabet case, the derivative $|T'|$ is uniformly bounded from above, and the
super-exponential shrinking forces the corresponding Hausdorff dimension to be
zero. 
For countable alphabets, however, derivatives may be unbounded, and very
small cylinders can arise from large symbols.
Thus the case $\gamma=\infty$
depends on finer information about the growth of the branches and is not
covered by the present argument.
\end{remark}

We also consider two related intersections. For $y\in E$, define the following shrinking target sets
\[
S[T,f,y]
=
\left\{
x\in E:
|T^n x-y|<e^{-S_nf(x)}
\text{ for infinitely many } n\in\mathbb N
\right\},
\]
and
\[
S(T,\psi,y)
=
\left\{
x\in E:
|T^n x-y|<\psi(n)
\text{ for infinitely many } n\in\mathbb N
\right\}.
\]

 Define
\[
N(T)
=
\left\{
x\in E:
\liminf_{n\to\infty}|T^n x-x|>0
\right\}.
\]

We are interested in the intersections $R(T,\psi)\cap N(T)$ and
$S(T,\psi,y)\cap\mathrm{Bad}(T,y)$.
The following theorem gives the corresponding results. We
distinguish two cases: $\limsup_{n\to \infty}\psi(n)=0$ and $\limsup_{n\to \infty}\psi(n)>0$.

\begin{theorem}\label{theorem-recurrent-shrinking-psi}
Let $T$ be an EMM on
$X=\bigcup_{i\in\Lambda}\overline{I(i)}$,
let $y\in E$, and let
$\psi:\mathbb N\to(0,\infty)$.
Then the following statements hold.
\begin{enumerate}[$(1)$]
    \item If $\limsup_{n\to\infty}\psi(n)=0$, then
    \[
    R(T,\psi)\cap N(T)
    =
    S(T,\psi,y)\cap\mathrm{Bad}(T,y)
    =
    \emptyset.
    \]

    \item If $\limsup_{n\to\infty}\psi(n)>0$, then
    \[
    \begin{aligned}
    \dim_{\mathrm H}\big(R(T,\psi)\cap N(T)\big)
    =
    \dim_{\mathrm H}\big(
    S(T,\psi,y)\cap\mathrm{Bad}(T,y)
    \big)
    =\dim_{\mathrm H}E.
    \end{aligned}
    \]
\end{enumerate}
\end{theorem}
We remark that the empty set results are evident. However, we include them in the theorem for completeness.

As a corollary, we obtain the following result.
\begin{corollary}\label{theorem-recurrent-shrinking-f}
Let $T$ be an EMM and let $y\in E$.
For every positive continuous function
$f:[0,1]\to(0,\infty)$, we have
\[
R[T,f]\cap N(T)
=
S[T,f,y]\cap\mathrm{Bad}(T,y)
=
\emptyset.
\]
\end{corollary}
\begin{proof}
Set $m=\min\{f(x):x\in [0,1]\}$ and $\psi(n)=e^{-mn}$.
Since $S_nf(x)\geq mn$, we have
$R[T,f]\subset R(T,\psi)$ 
 and 
$S[T,f,y]\subset S(T,\psi,y)$.
As $\psi(n)\to0$, the conclusion follows from Theorem~\ref{theorem-recurrent-shrinking-psi}.
\end{proof}

We end this subsection by briefly presenting the proof strategy and the organization
of the paper.
The upper bounds in
the above theorems follow immediately from the set inclusions and the known dimension
formulae for the corresponding recurrence sets. 
For instance,
\[
R[T,f]\cap\mathrm{Bad}(T,y)\subset R[T,f].
\]
The main task is proving the matching lower bounds.

To this end, we construct some Cantor type subsets of the desired intersections. 
The construction has three types of symbolic blocks:
\textit{free blocks}, which carry the pressure and determine the dimension; \textit{returning
blocks}, which force the orbit to return close to its initial point at selected
times; and \textit{safe blocks}, which prevent the orbit from entering a 
neighborhood of the given target.
The free blocks can be chosen sufficiently long that the proportion
of symbols in safe and connecting blocks is arbitrarily small. The selected
return times are sparse enough to control the additional inserted words.
Consequently, the Hausdorff dimension of the constructed Cantor set is governed
by the same pressure function as that of the quantitative recurrence set alone.

Our paper is organized as follows.
In Section~\ref{preliminaries}, we recall some basic facts on expanding Markov
maps,  cylinders and pressure functions. 
Section~\ref{section-cantor}
describes the general Cantor construction used in the lower bound arguments.
Section~\ref{proof-main-theorem-finite} gives a detailed proof of 
Theorem~\ref{theorem-f-finite}.
Section~\ref{proof-f-countable} proves the countable alphabet analogue,
Theorem~\ref{theorem-f-countable}. Sections~\ref{therorem1.3-proof} and
\ref{proof-psi-countable}
are devoted to proving
 Theorems~\ref{main-theorem-finite} and
\ref{theorem-countable}. 
Finally, Section~\ref{proof-theorem-re-sh-f} provides the proof of
Theorem~\ref{theorem-recurrent-shrinking-psi}.

\section{Preliminaries}\label{preliminaries}
\subsection{Symbolic notation and cylinder estimates}
In this subsection, we collect some notation and basic properties of expanding Markov maps that will be used throughout the paper.
Let $T$ be an expanding Markov map (EMM) defined on
$X=\bigcup_{i\in\Lambda}\overline{I(i)}\subset [0,1].$ 
For a finite word \(u=(i_1,\ldots,i_n)\in\Lambda^n\), define the associated
cylinder by
\[
I(u)=
I(i_1,\ldots,i_n)
=
I_n(i_1,\ldots,i_n)
:=
I(i_1)\cap T^{-1}(I(i_2))\cap\cdots\cap T^{-(n-1)}(I(i_n)).
\]
If \(I(u)\neq\emptyset\), then \(u\) is called admissible and \(I(u)\) is
called an \(n\)-th order cylinder.
We denote by \(\Sigma^n\) the collection of
all admissible words of length \(n\), and by
\(\Sigma^*
\)
the collection of all finite admissible words.
We also allow the empty word $\varnothing$, set
$\Sigma^0=\{\varnothing\}$ and $I(\varnothing)=(0,1)$,
and interpret any prefix of length zero as $\varnothing$. 
For \(u\in\Sigma^*\), we write
\(|u|\) for the length of \(u\), and \(|I(u)|\) for the length of the
corresponding cylinder.

By the mixing assumption in Definition~\ref{1.1}, there exists a fixed
word
\begin{equation}\label{con}
\bar{\mathtt u}=(j_1,\ldots,j_{m_0})\in\Sigma^{m_0}
\end{equation}
which can be used to connect arbitrary admissible words. More precisely, if
\(u,v\in\Sigma^*\), then the concatenation
\((u,\bar{\mathtt u},v)
\)
is admissible.

For \(u\in\Sigma^*\),  let
\(
u^\infty:=(u,u,u,\cdots)
\)
denote the infinite repetition of \(u\). For \(a>0\), let \(u^a\) denote
the prefix of \(u^\infty\) of length \(\lceil a|u|\rceil\), where
\(\lceil t\rceil\) denotes the smallest integer greater than or equal to
\(t\). In particular, if \(a\in\mathbb N\), then \(u^a\) is the
concatenation of \(a\) copies of \(u\).

For \(x\in E\), a sequence \((i_1,i_2,\ldots)\in\Lambda^{\mathbb N}\)
is called a symbolic representation of \(x\) if
\[
\bigcap_{n\geq1}\overline{I(i_1,\ldots,i_n)}=\{x\}.
\]
\begin{remark}\label{remark2-1}
The endpoints of cylinders
 may have more than one symbolic representation. 
 Such endpoints form a countable set. Hence,  excluding/including these  points is harmless for the dimension estimates in this
paper. 
Whenever a point has more than one representation, we fix one of them.
\end{remark} 
After fixing a symbolic representation, we write
\(x=(i_1,i_2,\ldots),
\)
and
\[
    x|_k^j=i_k,i_{k+1},\ldots,i_j,
    \qquad 1\leq k\leq j.
\]
We also write
\[
    I_n(x):=I(x|_1^n).
\]
For $w\in\Sigma^n$, if an $n$-th order cylinder is immediately adjacent to
$I(w)$ on the right, we denote its word by $w^+$. The word $w^-$ is defined
similarly for the cylinder immediately adjacent to $I(w)$ on the left. 
If the
corresponding adjacent cylinder does not exist, the symbol $w^+$ or $w^-$ is
omitted whenever a family of words is defined.

The following proposition collects the admissibility, distortion, and cylinder-length estimates
used later.
\begin{proposition}[\cite{yuan2022}]\label{Pro1.4} 
Let $T$ be an EMM. 
Then the following properties hold.
\begin{enumerate}[$(1)$]
    \item For any $u_1,u_2\in\Sigma^*$, the concatenation
    $  (u_1,\bar{\mathtt{u}},u_2)$ 
    is admissible, where $\bar{\mathtt{u}}$ is defined in (\ref{con}).
    \item For any $(i_1,i_2,\dots,i_n)\in\Sigma^n$ and $i_{n+1}\in\Lambda$, 
    \begin{equation}\label{1.2}
    (i_1,i_2,\dots,i_n,i_{n+1})\in\Sigma^* \iff (i_n,i_{n+1})\in\Sigma^*.
    \end{equation}
    \item If $(i_1,i_2,\dots,i_n)\in\Sigma^n$ with $n\geq2$, then for any $1\leq k\leq n-1$, $(i_k,i_{k+1})\in\Sigma^*$.
    \item 
    We obtain the following bounded distortion property on basic intervals: there exists a constant $\widetilde K\geq1$ such that for every $u\in\Sigma^n$ and all $x,y\in I(u)$,
    \begin{equation}\label{1.3}
        \widetilde K^{-1}
        \leq
        \frac{|(T^n)'(x)|}{|(T^n)'(y)|}
        \leq
        \widetilde K .
    \end{equation}
    \item There exists $K_1\geq1$ such that for any $(i_1,i_2,\dots,i_n)\in\Sigma^n$ with $n\geq1$,
    \begin{equation*}
    K_1^{-1}\leq|I(i_1,i_2,\dots,i_n)||(T^n)'(x)|\leq K_1.
    \end{equation*}
    \item If $\Lambda$ is finite, then there exist constants $\lambda_1\geq\lambda_2>1$ and $b\in\mathbb{N}$ such that for every $u\in\Sigma^n$ and every $x\in I(u)$,
    \begin{equation}\label{6}
    \lambda_2^{n-b}
    \leq
    |(T^n)'(x)|
    \leq
    \lambda_1^{n+b},
    \qquad
    \lambda_1^{-n-b}
    \leq
    |I(u)|
    \leq
    \lambda_2^{-n+b}.
    \end{equation}
    \item There exists a constant $c_1\geq1$ such that for any $u,v\in\Sigma^*$ with $uv\in\Sigma^*$,
    \begin{equation}\label{eq-two-cylinder}
        c_1^{-1}|I(u)|\,|I(v)|
        \leq
        |I(uv)|
        \leq
        c_1 |I(u)|\,|I(v)|.
    \end{equation}
\end{enumerate}
\end{proposition}
These estimates are standard consequences of the Markov property, the
mixing assumption and $C^{1+\delta}$ condition. 
They follow from the arguments
in Propositions 2.3--2.7 of \cite{yuan2022}.
The assumption
\(\mathcal L\left(\bigcup_{i\in\Lambda}\overline{I(i)}\right)=1
\)
is not used in the proofs of these estimates.

We shall also use the following mass distribution principle.
\begin{proposition}[\cite{yuan2022}]\label{pro-mass}
Let $T$ be an EMM with finite alphabet $\Lambda$. 
Let $U\subset E$, and let $\mu$ be a probability measure with $\mu(U)=1$. 
Suppose that there exist constants $c>0$ and $s>0$ such that for every sufficiently small cylinder $I(i_1,\ldots,i_n)$,
\[
\mu\big(I(i_1,\ldots,i_n)\big)
\leq
c |I(i_1,\ldots,i_n)|^s .
\]
Then
\(
\dim_{\mathrm H} U\geq s.
\)
\end{proposition}

\subsection{Pressure functions and finite subsystem approximations}\label{pressure-function}
For a function
\(\phi:\bigcup_{i\in\Lambda}I(i)\to\mathbb R\), define
\[
Z_n(T,\phi)
:=
\sum_{u\in\Sigma^n}
\sup_{x\in I(u)}
\exp\bigl(S_n\phi(x)\bigr),
\]
where
\(
S_n\phi(x)=\phi(x)+\phi(Tx)+\cdots+\phi(T^{n-1}x).
\)
For all \(n,m\in\mathbb N\),
\[
Z_{n+m}(T,\phi)
\leq
Z_n(T,\phi)Z_m(T,\phi).
\]
Thus $\{\log Z_n(T,\phi)\}_{n\geq1}$ is subadditive, so
$\lim_{n\to\infty}n^{-1}\log Z_n(T,\phi)$ exists.
The pressure of \(\phi\) is defined by
\begin{equation}\label{8.1}
P(T,\phi)
:=
\lim_{n\to\infty}
\frac{1}{n}\log Z_n(T,\phi)=\lim_{n\to\infty}
\frac{1}{n}\log \sum_{u\in\Sigma^n}
\sup_{x\in I(u)}
\exp\bigl(S_n\phi(x)\bigr).
\end{equation}
For a positive continuous function $f$, put
\[
    h=f+\log|T'|.
\]
For $t\geq0$, write
\begin{equation}\label{8.0}
D_n(t)
:=
\sum_{u\in\Sigma^n}
\sup_{x\in I(u)}
\exp\bigl(-tS_nh(x)\bigr).
\end{equation}
Then,
\[
P(T,-th)
=
\lim_{n\to\infty}
\frac{1}{n}\log D_n(t).
\]
The supremum over a cylinder may be replaced by the value at any  point
of that cylinder without changing the pressure. Indeed, if $u\in\Sigma^n$
and $x,z\in I(u)$, then
\[
|S_nh(x)-S_nh(z)|
\leq
|S_nf(x)-S_nf(z)|
+
\left|\log\frac{|(T^n)'(x)|}{|(T^n)'(z)|}\right|.
\]
The second term is uniformly bounded by Proposition~\ref{Pro1.4}.

For finite or countable $\Lambda$, conditions~{\rm (i)--(ii)}
in Definition~\ref{1.1} give
\[
|(T^n)'(x)|
\geq
\min\{1,\alpha^{K-1}\}\lambda^{\lfloor n/K\rfloor}
\]
on every $n$-th order cylinder. Since
$|T^n(I(u))|\leq1$, the mean value theorem implies that
\[
\sup_{u\in\Sigma^n}|I(u)|
\leq
\frac{\lambda^{-\lfloor n/K\rfloor}}
{\min\{1,\alpha^{K-1}\}}
\longrightarrow0
\qquad(n\to\infty).
\]

Let $\epsilon>0$. By the uniform continuity of $f$, there
exists $q\in\mathbb N$ such that
\[
\sup_{\substack{w\in\Sigma^q\\
\xi,\eta\in I(w)}}
|f(\xi)-f(\eta)|
<\epsilon.
\]
Let $n>q$, let
$u=(i_1,\ldots,i_n)\in\Sigma^n$, and let $x,z\in I(u)$.
For $0\leq j\leq n-q-1$, we have
\(T^jx,T^jz
\in
I(i_{j+1},\ldots,i_{j+q}),
\)
and hence
\(|f(T^jx)-f(T^jz)|<\epsilon.
\)
Consequently,
\begin{align*}
|S_nf(x)-S_nf(z)|
&\leq
\sum_{j=0}^{n-q-1}
|f(T^jx)-f(T^jz)|
+
\sum_{j=n-q}^{n-1}
|f(T^jx)-f(T^jz)|
\\
&\leq
(n-q)\epsilon
+
2q\max_{\xi\in[0,1]}|f(\xi)|.
\end{align*}
It follows that
\[
\begin{aligned}
\sup_{\substack{u\in\Sigma^n\\x,z\in I(u)}}
|S_nh(x)-S_nh(z)|
\leq
(n-q)\epsilon
+
2q\max_{\xi\in[0,1]}|f(\xi)|
+
\log\widetilde K.
\end{aligned}
\]
Therefore,

\begin{equation}\label{eq-uniform-variation}
\sup_{\substack{u\in\Sigma^n\\x,z\in I(u)}}
|S_nh(x)-S_nh(z)|
=o(n).
\end{equation}

Therefore, if \(z_u\in I(u)\) is chosen arbitrarily for each
\(u\in\Sigma^n\), then
\[
\lim_{n\to\infty}\frac1n\log D_n(t)
=
\lim_{n\to\infty}\frac1n\log
\sum_{u\in\Sigma^n}e^{-tS_nh(z_u)}.
\]
Thus, the pressure is independent of the choice of the points \(z_u\).

We now recall the facts needed later.
We separate the finite and
countable alphabet cases.
\subsubsection{Finite alphabets}
Assume that $\Lambda$ is finite.
Let \(f:[0,1]\to(0,\infty)\) be
continuous and set \(h=f+\log|T'|\).
Let \(t(f)\) denote the unique real root of
\[
P(T,-t h)=0.
\]

For all $n>b$, we define $t_n(f)$ by
\begin{equation}\label{eq-tn-f}
\sum_{u\in\Sigma^n}e^{-t_n(f)S_nh(y_u)}=1,
\end{equation}
where $y_u$ denotes the midpoint of $I(u)$ and $b$ is defined in \eqref{6}.
For such $n$, we have $S_nh(y_u)>0$.
The sum in \eqref{eq-tn-f} is strictly decreasing from a value greater than one to zero, yielding existence and uniqueness.

\begin{proposition}\label{2.4}
Let $f$ be a positive continuous function defined on \([0,1]\), and let  $t_n(f)$ be as in \eqref{eq-tn-f}. Then
\[
\lim_{n\to\infty} t_n(f)=t(f).
\]
\end{proposition}
The convergence in Proposition~\ref{2.4} follows from the definition of
pressure function, Proposition~\ref{Pro1.4}, and the uniqueness of the zero of
$t\mapsto P\bigl(T,-t(f+\log|T'|)\bigr)$.

Let $\psi:\mathbb N\to(0,\infty)$ be a positive function and recall that 
\[
\log\gamma
=
\liminf_{n\to\infty}
\frac{-\log\psi(n)}{n}.
\]

For $1\leq\gamma<\infty$, let $t(\Lambda,\gamma)$ denote the
unique zero of $P(T,-t(\log\gamma+\log|T'|))$  when $\Lambda$ is finite.
For $n>b$, define $t_n(\gamma)$ to be the unique real root of
\begin{equation}\label{eq-tn-gamma}
\sum_{u\in\Sigma^n}
e^{-tS_n(\log\gamma+\log|T'|)(y_u)}
=
1,
\end{equation}
where  \(y_u\) is the midpoint  in \(I(u)\) and $b$ is defined in \eqref{6}. Equivalently, since
\(\log\gamma\) is constant,
\begin{equation}\label{eq-sn}
    S_n(\log\gamma+\log|T'|)(y_u)
=
n\log\gamma+\log |(T^n)'(y_u)|.
\end{equation}
As in Proposition~\ref{2.4},
$\lim_{n\to \infty}t_n(\gamma)= t(\Lambda,\gamma)$. 
\subsubsection{Countable alphabets}
Assume that \(\Lambda\) is countable.
Let $A\subset\Lambda$ be finite,
and denote by $E_A$ the set of points in $E$
whose symbolic representations use only symbols in $A$.
By the full image property, every word \(w\in A^{n}\) is admissible, where \(n\in \mathbb{N}.\)
For a function
\(\phi:\bigcup_{i\in\Lambda}I(i)\to\mathbb R,
\)
define
\[
Z_{n,A}(T,\phi)
:=
\sum_{w\in A^n}
\sup_{x\in I(w)}
\exp\bigl(S_n\phi(x)\bigr).
\]
For all \(n,m\in\mathbb N\),
\[
Z_{n+m,A}(T,\phi)
\leq
Z_{n,A}(T,\phi)Z_{m,A}(T,\phi).
\]
Consequently, the limit
\(\lim_{n\to\infty}\frac1n\log Z_{n,A}(T,\phi)
\)
exists, and we define
\[
P_A(T,\phi)
:=
\lim_{n\to\infty}\frac1n\log Z_{n,A}(T,\phi).
\]
We use the conventions $Z_{n,\Lambda}(T,\phi)=Z_n(T,\phi)$
and $P_\Lambda(T,\phi)=P(T,\phi)$. For finite $A$ with at least
two symbols and the functions $h$ used below,
\eqref{eq-uniform-variation} shows that restricting a supremum
to $I(w)\cap E_A$ does not change the limiting pressure.
Thus $P_A(T,-th)$ is also the pressure of the finite-alphabet
restriction of the original map $T$.

For a positive continuous function
\(f:[0,1]\to(0,\infty)\), define
\[
t(A,f)
:=
\inf\left\{
t\geq0:
P_A\bigl(T,-t(f+\log|T'|)\bigr)\leq0
\right\}
\]
and
\[
t(\Lambda,f)
:=
\inf\left\{
t\geq0:
P\bigl(T,-t(f+\log|T'|)\bigr)\leq0
\right\}.
\]
For $x_w\in I(w)$, also set
\[
E_{n,A,f}(t)
:=
\sum_{w\in A^n}
e^{-tS_n(f+\log|T'|)(x_w)}.
\]
Equation~\eqref{eq-uniform-variation} shows that the choice
of $x_w$ does not affect the limits in Proposition~\ref{6.5}.

The following propositions summarize the finite subsystem approximation
properties of pressure in the countable alphabet case.

\begin{proposition}[\cite{Hanus2002,Li2014}]
\label{6.1}
Let \(f:[0,1]\to(0,\infty)\) be a continuous function, and let \(T\) be an EMM with
countable alphabet \(\Lambda\). Set \(h=f+\log|T'|\). For every $t\geq0$, we have
\[
P_{\Lambda}(T,-th)
=
\sup\left\{
P_A(T,-th):
A \text{ is a finite subset of } \Lambda
\right\}.
\]
Consequently,
\(t(\Lambda,f)
=
\sup\left\{
t(A,f):
A \text{ is a finite subset of } \Lambda
\right\}.
\)
\end{proposition}
\begin{proposition}[\cite{Hanus2002,Li2014}]
\label{6.5}
Let \(f:[0,1]\to(0,\infty)\) be a continuous function, and let \(T\) be an EMM with
countable alphabet \(\Lambda\). For \(A_M=\{1,\ldots,M\}\subset\Lambda\), define
\[
t_n(A_M,f)
=
\inf\{t\geq0:E_{n,A_M,f}(t)\leq1\}.
\]
Then
\begin{equation}\label{6.2}
\lim_{n\to\infty}t_n(A_M,f)=t(A_M,f), \quad \text{and}
\quad
\lim_{M\to\infty}t(A_M,f)=t(\Lambda,f).
\end{equation}
\end{proposition}

We explain why Propositions~\ref{6.1} and~\ref{6.5} apply
with $K\geq1$. Put $A_N=\{1,\ldots,N\}$.
For $T^K$, use the intervals $I(i_1,\ldots,i_K)$ as its branch
intervals. Each of them is mapped onto $(0,1)$ by the full image
property. The derivative of $T^K$ on each of these intervals has
absolute value at least $\lambda$, so the derivative of its
inverse has absolute value at most $\lambda^{-1}<1$.

A sequence of $n$ such branch choices specifies exactly $nK$
original symbols. Also,
\[
S_Kh(x)=S_Kf(x)+\log|(T^K)'(x)|,\qquad
\sum_{j=0}^{n-1}S_Kh(T^{jK}x)=S_{nK}h(x).
\]
Consequently, for $A=A_N$ or $A=\Lambda$,
\[
\sum_{u\in A^{nK}}\sup_{x\in I(u)}
\exp\left(-t\sum_{j=0}^{n-1}S_Kh(T^{jK}x)\right)
=Z_{nK,A}(T,-th).
\]

Taking logarithms, dividing by $n$, and passing to the
limit, we obtain
\[
P(T^K,-tS_Kh)=K P(T,-th).
\]
Similarly, restricting the branches of $T^K$ to
$I(u)$ with $u\in A_N^K$ gives the pressure
$K P_{A_N}(T,-th)$.

By the pressure approximation for
$T^K$ in \cite{Yuan2023}, its pressure is the
supremum of the pressures obtained from finite
collections of its branches.
Any finite collection of length-$K$ words uses only
finitely many symbols from $\Lambda$.
Hence, for some $N$, all these words belong to $A_N^K$.
For every $n$, the corresponding partition sum is
therefore bounded above by $Z_{nK,A_N}(T,-th)$.
Conversely, $A_N^K$ is itself a finite collection
of length-$K$ words. Consequently,
\[
K P(T,-th)
=
\sup_{N\geq2}K P_{A_N}(T,-th),
\]
and hence
\[
P(T,-th)=\sup_{N\geq2}P_{A_N}(T,-th).
\]

For each fixed $N$, consider the original map $T$
on the intervals $I(i)$, $i\in A_N$.
It satisfies conditions~{\rm (i)--(iii)} of
Definition~\ref{1.1} with the same $K$.
The full image property gives the mixing condition
with $m_0=0$, while \eqref{1.5} ensures that
$|T'|$ is bounded above on these finitely many
intervals.
Thus Proposition~\ref{Pro1.4} applies, with constants
that may depend on $N$, and $E_{A_N}$ is invariant
under $T$.
In the subsequent Cantor constructions, we use
words over $A_N$ for the map $T$.

For $1\leq\gamma<\infty$ and finite $A\subset\Lambda$, define
\[
t(A,\gamma)
:=
\inf\left\{t\geq0:
P_A\bigl(T,-t(\log|T'|+\log\gamma)\bigr)\leq0
\right\}.
\]
For the full alphabet, define
\[
t(\Lambda,\gamma)
:=
\inf\left\{t\geq0:
P\bigl(T,-t(\log|T'|+\log\gamma)\bigr)\leq0
\right\}.
\]
\begin{proposition}[\cite{Yuan2023}]
\label{0.4}
Let \(T\) be an EMM with finite or countable alphabet \(\Lambda\). 
The function
\(\gamma\longmapsto t(\Lambda,\gamma)
\)
is continuous on $[1,\infty)$. In particular,
\[
\lim_{\gamma\to1}t(\Lambda,\gamma)=t(\Lambda,1).
\]
\end{proposition}

The same pressure argument applies to
$\log\gamma+\log|T'|$ for $1\leq\gamma<\infty$, including
$\log|T'|$ when $\gamma=1$. In particular,
\(\lim_{N\to\infty}t(A_N,\gamma)=t(\Lambda,\gamma).
\)
The pressure formula for Hausdorff dimension, applied to the
inverse branches of $T^K$, also gives
\[
\dim_{\mathrm H}E_{A_N}=t(A_N,1),\qquad
\lim_{N\to\infty}\dim_{\mathrm H}E_{A_N}
=t(\Lambda,1)=\dim_{\mathrm H}E;
\]
see \cite{Hanus2002,Li2014,Yuan2023}.

\section{The Cantor Construction: General Principles}
\label{section-cantor}
The lower bounds are proved using Cantor constructions in symbolic space.
These constructions use the following three types of blocks.
The free blocks are admissible words that can be chosen independently
at each stage. 
These choices are used to define the probability measure and
determine the Hausdorff dimension. The returning blocks are
copies of initial words inserted after prescribed times \(n_k\); their lengths
are chosen so that the required conditions
hold. 
The auxiliary blocks consist of the connecting blocks needed for
admissibility and the fixed words inserted to keep the orbit away from the
target.

The returning blocks and the auxiliary blocks play different roles. 
The returning blocks are generally not negligible: their contribution gives rise to the additional term \(f\) or \(\log\gamma\) in the corresponding pressure function.
The proportion of symbols occupied by auxiliary blocks can be made
arbitrarily small by taking the free blocks sufficiently long.
Hence they do not affect the limiting dimension.

The purpose of this section is to give the elementary principles used in the
constructions. 
These principles will be implemented in detail in the proof of
Theorem~\ref{theorem-f-finite}. In the later proofs we will only indicate the
necessary modifications.
\subsection{Safe blocks and avoidance}
\label{subsection-safe-blocks}
Safe blocks are used to ensure that the orbit remains a positive distance
from the target. 
In each application, we choose a finite set
\(\mathcal U\subset\Sigma^{\ell}
\)
and a word
\(v\in\Sigma^{\ell}\setminus\mathcal U.
\)
The set \(\mathcal U\) is chosen so that, if \(T^n x\) is sufficiently close
to the target \(y\), then
\(x|_{n+j}^{n+j+\ell-1}\in\mathcal U
\)
for every \(j\) in a prescribed finite range. The safe block \(v\) is
inserted so that, for every sufficiently large \(n\), there exists such a
\(j\) satisfying
\(x|_{n+j}^{n+j+\ell-1}=v.
\)
Since \(v\notin\mathcal U\), \(T^n x\) cannot be sufficiently close to
\(y\). 
Hence, the orbit eventually remains a positive distance from \(y\).

\subsection{Returning blocks}
\label{subsection-returning-blocks}

The following observation is used to force recurrence.

\begin{lemma}\label{lemma-returning-block}
Let \(x\in E\) have a symbolic representation
\(x=(i_1,i_2,\ldots).
\)
Let \(n,m\in\mathbb N\). Suppose 
\((x|_1^m)=(x|_{n+1}^{n+m}).
\)
Then
\(|T^n x-x|\leq |I(x|_1^m)|.
\)
\end{lemma}

\begin{proof}
The equality \((x|_1^m)=(x|_{n+1}^{n+m})\) implies that both \(x\) and \(T^n x\)
belong to the same \(m\)-th order cylinder \(I(x|_1^m)\). 
Hence their distance
is bounded from above by \(|I(x|_1^m)|\).
\end{proof}

At a prescribed return time \(n\), the construction repeats a sufficiently
long initial block after the \(n\)-th symbol. When the recurrence condition is
given by
\[
|T^n x-x|<e^{-S_n f(x)},
\]
the length of the repeated block is chosen so that its cylinder has length
smaller than \(e^{-S_n f(x)}\). This gives rise to the pressure function
\(
P\bigl(T,-t(f+\log|T'|)\bigr).
\)
When the recurrence condition is given by
\[
|T^n x-x|<\psi(n),
\]
the repeated block is chosen so that its cylinder has length smaller than
\(\psi(n)\). For a subsequence satisfying
\[
\frac{-\log\psi(n)}{n}\to\log\gamma,
\]
the corresponding pressure function is
\(
P\bigl(T,-t(\log\gamma+\log|T'|)\bigr).
\)

\subsection{Choice of return times and mass distribution}
\label{subsection-sparse-principle}
Let \(\{n_k\}_{k\geq1}\) be the return times. Between consecutive return times, free and auxiliary blocks are inserted
at prescribed positions. 
The lengths of the
free blocks are chosen sufficiently large so that the proportion of symbols
belonging to the auxiliary blocks can be made arbitrarily small.

The return times \(n_k\) are chosen sufficiently sparse. The returning blocks
are treated separately, since their lengths are chosen according to
\(e^{-S_{n_k}f(x)}\) or \(\psi(n_k)\), depending on the theorem under
consideration.

The probability measure is assigned through the choices of the free blocks.
Once the free blocks have been chosen, the returning and auxiliary blocks are
fixed. We then prove an estimate of the form
\[
\mu(I_n(x))\leq C|I_n(x)|^s.
\]
The calculation is given in detail in the proof of
Theorem~\ref{theorem-f-finite}.

\section{Recurrence with \texorpdfstring{$e^{-S_n f}$}{exp(-S_n f)}: Finite Alphabets}
\label{proof-main-theorem-finite}
In this section, we prove Theorem~\ref{theorem-f-finite}. Since
\(R_{\mathrm{Bad}}[T,f,y]\subset R[T,f],
\)
the upper bound follows from the known formula for \(\dim_{\mathrm H}R[T,f]\) in \cite{Yuan2023}.
It remains to show
\[
\dim_{\mathrm H}R_{\mathrm{Bad}}[T,f,y]\geq t(f),
\]
where \(t(f)\) is the unique root of
\(P(T,-t(f+\log|T'|))=0.
\)

Assume throughout this section that \(\Lambda=\{1,\ldots,d\}\), and fix a
symbolic representation \(y=(y_1,y_2,\ldots)\). We construct a Cantor subset
following the framework of Section~\ref{section-cantor}.

\subsection{Choice of parameters and the safe block}
Let \(m_0\) be the integer appearing in the mixing condition of
Definition~\ref{1.1}. Choose 
\(1<a<d^{1/(m_0+1)}
\)
and  \(m^*\in \mathbb{N}\) such that
\[
\big(d^{1/(m_0+1)}\big)^{m^*}>9,
\]
where   $d=\#\Lambda$.
If $t(f)=0$, the required lower bound is immediate.
Therefore, we  assume
$t(f)>0$.
Given $\varepsilon>0$, choose
$\delta>0$ so small that
\(3\delta+\frac{6\delta}{\log\lambda_2}
<\min\{\varepsilon,t(f)\}.
\)
By Proposition~\ref{2.4}, there exists \(N_0\in\mathbb N\) with $N_0>b$
such that, for every \(n\geq N_0\),
\begin{equation}\label{eq-tnf-new}
|t_n(f)-t(f)|\leq \delta .
\end{equation}
Moreover, since \(f\) is continuous and the diameters of cylinders tend to
zero, by increasing \(N_0\) if necessary, we may assume that for every
\(n\geq N_0\), every \(w\in\Sigma^n\), and all \(x,z\in I(w)\),
\begin{equation}\label{eq-S-nf-new}
|S_nf(x)-S_nf(z)|\leq n\delta .
\end{equation}

Choose \(M\) sufficiently large such that
$M-N_0>3(\log_a M+9m_0+m^*+\|f\|)
$ 
and
\begin{equation*}
K_1c_1^2\lambda_1^{\log_a(M+1)+m^*+1+2m_0+b}
\leq\lambda_2^{M\delta/3},
\quad
\frac{\|f\|(\log_a(M+1)+m^*+1+2m_0)}{M}\leq\delta.
\end{equation*}
where \(K_1,\lambda_2\) and \(c_1\) are the constants in
Proposition~\ref{Pro1.4}
and $\|f\|:=\max\{|f(x)|:x\in [0,1]\}$.
Set
\begin{equation}\label{eq-l}
    \ell=\lfloor\log_a(M+1)\rfloor+m^*+1.
\end{equation}
Then
\begin{equation}\label{eq-M-new}
M-\ell-2m_0\geq N_0,\qquad M>2\ell+3m_0,
\end{equation}
and
\begin{equation}\label{eq-M-distortion-new}
\begin{gathered}
K_1c_1^2\lambda_1^{\ell+2m_0+b}\leq\lambda_2^{M\delta/3},\quad
\frac{\|f\|(\ell+2m_0)}{M}\leq\delta.
\end{gathered}
\end{equation}

The choice of \(\ell\) gives enough admissible words of length \(\ell\) to avoid
the finite family \(\mathcal U(M,\ell)\) defined below.
Indeed,
for sufficiently large $M$,
\[
\#\Sigma^\ell
\geq d^{\lfloor\ell/(m_0+1)\rfloor}
\geq d^{-1}d^{\ell/(m_0+1)}
>9(M+1).
\]
The last inequality follows from \eqref{eq-l} and
$\log d/((m_0+1)\log a)>1$.

The return times are chosen successively during the construction.
At every stage, require
\begin{equation}\label{eq-sparse-birkhoff-corrected}
n_k\geq2\sum_{j=1}^{k-1}n_j
\end{equation}
and
\begin{equation}\label{eq-sparse-last-new}
e^{6\delta b}c_1^{3k}\lambda_1^{k(2M+m_0+2b+1)}
\leq\lambda_2^{2\delta n_k/3-b\delta}.
\end{equation}
Choose $n_1\geq\max\{N_0,3M\}$ sufficiently large that these
conditions hold for $k=1$ and
\[
\|f\|(2M+m_0)\leq\delta n_1.
\]
For $k=1$, the sum in
\eqref{eq-sparse-birkhoff-corrected} is interpreted as zero.

We now choose the safe word. 
Define the finite forbidden family
\[
\mathcal U(M,\ell)
=
\left\{
y|_i^{i+\ell-1},
\left((y|_1^{i+\ell-1})^+\right)|_i^{i+\ell-1},
\left((y|_1^{i+\ell-1})^-\right)|_i^{i+\ell-1}
:
1\leq i\leq 3M
\right\},
\]
where $\ell$ is defined in \eqref{eq-l}.
If an adjacent cylinder in the definition of $\mathcal U(M,\ell)$ does not
exist, the corresponding term is omitted.
Since
\[
\#\mathcal U(M,\ell)\leq 9M
\quad\text{and}\quad
\#\Sigma^{\ell}>9(M+1),
\]
we may choose
\[
v\in\Sigma^{\ell}\setminus\mathcal U(M,\ell).
\]
The role of the word $v$ is to
prevent the orbit from entering a  neighbourhood of the target \(y\).

\subsection{Construction of the Cantor set}
Let
\[\bar{\mathtt v}=(\bar{\mathtt u},v,\bar{\mathtt u})
\in\Sigma^{\ell+2m_0},
\]
where \(\bar{\mathtt u}\) is the connecting word defined in \eqref{con}.
If $m_0=0$, 
$(\bar{\mathtt u}^{\infty})|_1^q$ is replaced by
$(v|_1^1)^q$.

The construction has three ingredients. Free blocks of length
\(M-\ell-2m_0
\)
carry the pressure and the measure. Copies of \(\bar{\mathtt v}\) are inserted
as safe blocks to enforce the avoidance of the target.
Finally, at each position \(n_k\), a returning block is appended so that the orbit returns to its
initial point.

Choose \(p_1\in\mathbb N\) and an integer \(q_1\) with \(0\leq q_1<M\) such that 
\(n_1=(p_1+1)M+q_1.
\)
Define
\[
\mathcal F_1
=
\left\{
\left(
(v,\bar{\mathtt u})^{\infty}|_1^{M+q_1-\ell-2m_0},
\bar{\mathtt v},
u^{(1)}_1,\bar{\mathtt v},
u^{(1)}_2,\ldots,
u^{(1)}_{p_1},\bar{\mathtt v}
\right)
:
u^{(1)}_i\in\Sigma^{M-\ell-2m_0},\ 1\leq i\leq p_1
\right\}.
\]
Each free word in $\mathcal F_1$ is followed by
$\bar{\mathtt v}$. 
These words are admissible by Proposition~\ref{Pro1.4}.
When $m_0>0$, every $W_1\in\mathcal F_1$ ends with the
connecting word $\bar{\mathtt u}$; when $m_0=0$, admissible
words can be concatenated directly. Hence $W_1^\infty$ is
admissible in either case. The same argument applies to
$W_k^\infty$, since $W_k$ ends with $\bar{\mathtt v}$.

For each $W_1\in\mathcal F_1$, let $z(W_1)$ be the midpoint of
$I_{n_1}(W_1)$. 
There is a smallest $m_1(W_1)\in\mathbb N$ such that
\[
\left|I_{m_1(W_1)}\bigl((W_1)^\infty|_1^{m_1(W_1)}\bigr)\right|
<
e^{-S_{n_1}f(z(W_1))-n_1\delta}.
\]
By the minimality of $m_1(W_1)$,
\(
e^{-S_{n_1}f(z(W_1))-n_1\delta}
\leq
\big|I_{m_1(W_1)-1}\bigl((W_1)^\infty|_1^{m_1(W_1)-1}\bigr)\big|.
\)
Set
\[
    \eta_1(W_1)=1+\frac{m_1(W_1)}{n_1},
\]
then \(n_1\eta_1(W_1)=n_1+m_1(W_1).\)
Define  the first-level family
\[
\mathcal C_1
=
\left\{
V_1(W_1):=(W_1)^{\eta_1(W_1)}:
W_1\in\mathcal F_1
\right\}.
\]

Now let $k\geq2$ and suppose that $\mathcal C_{k-1}$ has been
constructed. 
It is a finite family of finite words.
Choose one integer
$n_k$ satisfying
\eqref{eq-sparse-birkhoff-corrected}, \eqref{eq-sparse-last-new},
\begin{equation}\label{eq-markov-nk2-new}
\begin{aligned}
n_k\geq k\max_{V\in\mathcal C_{k-1}}|V|,\quad
n_k\geq\max_{V\in\mathcal C_{k-1}}|V|+3M+m_0,
\end{aligned}
\end{equation}
and
\[
\|f\|\left(\max_{V\in\mathcal C_{k-1}}|V|+2M+m_0\right)
\leq\delta n_k.
\]
For each
$V_{k-1}=V_{k-1}(W_{k-1})\in\mathcal C_{k-1}$,
let $j$ be the least positive integer such that
\[
(W_{k-1}^{\infty})|_{|V_{k-1}|+j-\ell+1}^{|V_{k-1}|+j}=v.
\]
Define
\[
v_k(V_{k-1})=
\begin{cases}
\displaystyle
\left(
(W_{k-1}^{\infty})|_{|V_{k-1}|+1}^{|V_{k-1}|+j},
(\bar{\mathtt u}^{\infty})|_1^{M-\ell-2m_0-j}
\right),
& j\leq\ell+m_0,\\[2mm]
\displaystyle
\left(
\bar{\mathtt u},v,
(\bar{\mathtt u}^{\infty})|_1^{M-2\ell-3m_0}
\right),
& j>\ell+m_0.
\end{cases}
\]
In both cases,
\[
|v_k(V_{k-1})|=M-\ell-2m_0.
\]
Once $V_{k-1}$ is fixed, this word is uniquely determined.
For each \(V_{k-1}\in\mathcal C_{k-1}\), choose
\(p_k(V_{k-1})\in\mathbb N\) and
\(0\leq q_k(V_{k-1})<M\) such that
\[
n_k-|V_{k-1}|-m_0
=
\bigl(p_k(V_{k-1})+1\bigr)M+q_k(V_{k-1}).
\]
The second inequality in \eqref{eq-markov-nk2-new} ensures
$p_k(V_{k-1})\geq1$.

Put $p_k:=p_k(V_{k-1})$ and $q_k:=q_k(V_{k-1})$.
Define $\mathcal F_k$ to be the family of all words
\[
\begin{aligned}
(&V_{k-1},
v_k(V_{k-1}),\bar{\mathtt v},
(\bar{\mathtt u}^{\infty})|_1^{q_k},
\bar{\mathtt u},
u^{(k)}_1,\bar{\mathtt v},
u^{(k)}_2,\bar{\mathtt v},
\ldots,
u^{(k)}_{p_k},\bar{\mathtt v}),
\end{aligned}
\]
where
\(V_{k-1}\in\mathcal C_{k-1}\),
\(u_i^{(k)}\in\Sigma^{M-\ell-2m_0}\),
and 
\(1\leq i\leq p_k(V_{k-1}).
\)
For clarity, Figure~\ref{fig-Fk-structure} shows the symbolic structure of a
typical word in \(\mathcal F_k\).
\begin{figure}
    \centering
    \includegraphics[width=0.8\linewidth]{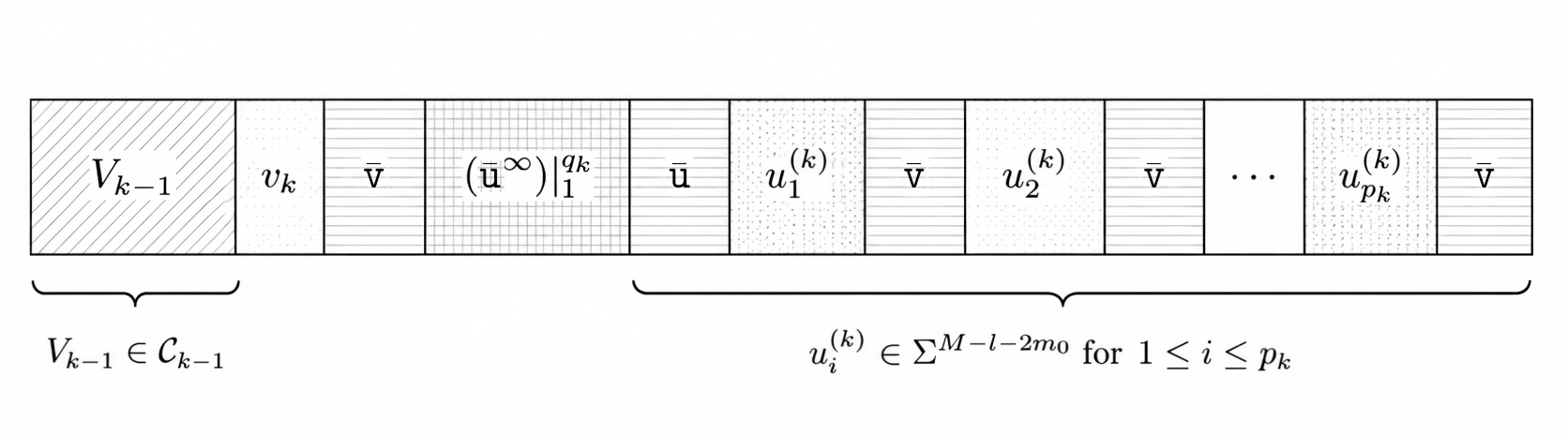}
    \caption{Structure of a word in \(\mathcal F_k\).}
    \label{fig-Fk-structure}
\end{figure}

For each $W_k\in\mathcal F_k$, let $z(W_k)$ be the midpoint of
$I(W_k)$, and write $z=z(W_k)$ in the defining inequality below.

For each \(W_k\in\mathcal F_k\), let
\(m_k(W_k)\) be the smallest positive integer such that
\begin{equation}\label{eq-markov-m-k-11-new}
\left|
I_{m_k(W_k)}
\bigl((W_k)^\infty|_1^{m_k(W_k)}\bigr)
\right|
<
e^{-S_{n_k}f(z)-n_k\delta}.
\end{equation}
The minimality of $m_k(W_k)$ also gives
\begin{equation}\label{eq-markov-m-k-lower-new}
e^{-S_{n_k}f(z(W_k))-n_k\delta}
\leq
\left|
I_{m_k(W_k)-1}\bigl((W_k)^\infty|_1^{m_k(W_k)-1}\bigr)
\right|.
\end{equation}
Set
\(\eta_k(W_k)
=
1+\frac{m_k(W_k)}{n_k}.
\)
Since
\(S_{n_k}f(x)\leq n_k\|f\|\)
and
\(|I_m(x)|\leq\lambda_2^{-m+b},
\)
the minimality of $m_k(W_k)$ gives
\[
m_k(W_k)
\leq
\frac{\|f\|+\delta}{\log\lambda_2}n_k+b+1.
\]
For $k\geq2$, the first condition in \eqref{eq-markov-nk2-new} gives
\[
\sup_{W_{k-1}\in\mathcal F_{k-1}}
\frac{n_{k-1}+m_{k-1}(W_{k-1})}{n_k}\leq\frac1k.
\]
Hence,
\(n_{k-1}+m_{k-1}( W_{k-1})=o(n_k)
\)
 for any $W_{k-1}\in \mathcal{F}_{k-1}$.
Define
\[
\mathcal C_k
=
\left\{
V_k(W_k):=(W_k)^{\eta_k(W_k)}:
W_k\in\mathcal F_k
\right\}.
\]
For $V_k=V_k(W_k)$, write $m_k(V_k):=m_k(W_k)$. The limiting set is
\[
C_\infty
=
\bigcap_{k=1}^\infty
\bigcup_{V_k\in\mathcal C_k}I(V_k).
\]
We remark that
for each \(k\), the integer \(n_k\) does not depend on the choice of
\(W_k\in\mathcal F_k\), whereas \(m_k(W_k)\) may depend on \(W_k\).
For every \(x\in C_\infty\), let \(V_k(x)\in\mathcal C_k\) be the
word such that
\(x\in I(V_k(x)),
\)
and let \(W_k(x)\in\mathcal F_k\) be the word satisfying
\(V_k(x)=V_k(W_k(x)).
\)
Define
\[
m_k(x):=m_k(W_k(x)).
\]
For \(k\geq2\), let \(p_k(x)\) and \(q_k(x)\) be the integers associated
with \(V_{k-1}(x)\) in the definition of \(\mathcal F_k\).
When \(x\) is fixed, we write \(m_k,p_k,q_k\) for
\(m_k(x),p_k(x),q_k(x)\).

We prove by induction that, in $W_k^\infty$, the distance
between the starting positions of any two consecutive
occurrences of $v$ is at most $2M+m_0$.

For $k=1$, the starting positions of $v$ in the words
$\bar{\mathtt v}$ appearing in $W_1$ are separated by $M$.
Across two consecutive repetitions of $W_1$, the
corresponding distance is $M+q_1<2M$.
Any additional occurrences of $v$ can only decrease
the maximum distance between consecutive starting positions.
Suppose the assertion holds for $W_{k-1}^\infty$.
Let $j$ be as in the definition of $v_k(V_{k-1})$.
If $j\leq\ell+m_0$, the distance between the starting
positions of the occurrence of $v$ completed along
$W_{k-1}^\infty$ and the occurrence of $v$ in the
$\bar{\mathtt v}$ following $v_k(V_{k-1})$ is
\(M-m_0-j<M.
\)
If $j>\ell+m_0$, the distance between the occurrence
of $v$ inserted in $v_k(V_{k-1})$ and the occurrence
of $v$ in the following $\bar{\mathtt v}$ is
\(M-\ell-2m_0<M.
\)
The distance between the occurrences of $v$ in the
$\bar{\mathtt v}$ following $v_k(V_{k-1})$ and
the $\bar{\mathtt v}$ following $u_1^{(k)}$ is
\(M+q_k+m_0<2M+m_0.
\)
The occurrences of $v$ following consecutive free
words $u_i^{(k)}$ are separated by $M$.

Finally, since $W_k$ begins with $W_1$ and ends with
$\bar{\mathtt v}$, the distance from the occurrence
of $v$ in the last $\bar{\mathtt v}$ of $W_k$ to
the occurrence of $v$ in the first $\bar{\mathtt v}$
of the next repetition of $W_k$ is $M+q_1<2M$.
This completes the induction.

Every finite prefix of a constructed infinite sequence is
contained in some $W_k$. The initial part contains $v$ at
position one. Consequently, every block of length
$2M+\ell+m_0$ in such a sequence contains a complete occurrence
of $v$.

\begin{lemma}\label{lemma-subset-new}
We have
\(C_\infty\subset R_{\mathrm{Bad}}[T,f,y].
\)
\end{lemma}

\begin{proof}
Let $x\in C_\infty$. 
We first prove that $x\in\mathrm{Bad}(T,y)$.
By the construction above, for every
sufficiently large $n$, the word
\(x|_{n+1}^{n+2M+\ell+m_0}
\)
contains a complete occurrence of $v$. Consequently, there
exists an integer
\(1\leq j\leq2M+m_0+1\leq3M
\)
such that
\[
x|_{n+j}^{n+j+\ell-1}=v.
\]
We now use $v\notin\mathcal U(M,\ell)$ to obtain a uniform lower bound for the
distance to $y$.
If \(T^n x\) were sufficiently close to \(y\), then, by the definition
of \(\mathcal U(M,\ell)\), for every \(1\leq i\leq 2M+m_0+1\),
\(x|_{n+i}^{n+i+\ell-1}\in\mathcal U(M,\ell).
\)
This is because the
order \(i+\ell-1\) cylinder containing \(T^n x\) would have to coincide with the
cylinder \(I(y|_1^{i+\ell-1})\) or one of its adjacent cylinders. 
Hence, each word
\((x|_{n+i}^{n+i+\ell-1})
\)
would belong to \(\mathcal U(M,\ell)\). 
This contradicts the occurrence of
\(v\notin\mathcal U(M,\ell)\).
Therefore, there exists \(a>0\) such that
\[
|T^n x-y|\geq a
\qquad\text{for all sufficiently large } n,
\]
and hence \(x\in\mathrm{Bad}(T,y)\).

It remains to verify the recurrence condition. At the prescribed return time
\(n_k\), the block following position \(n_k\) repeats the initial block of
length \(m_k\). Hence, by Lemma~\ref{lemma-returning-block} and
 \eqref{eq-markov-m-k-11-new}
\[
|T^{n_k}x-x|
<
e^{-S_{n_k}f(z)-n_k\delta},
\]
where \(z\) is the midpoint of the corresponding \(n_k\)-th cylinder. 
By \eqref{eq-S-nf-new},
\(S_{n_k}f(x)\leq S_{n_k}f(z)+n_k\delta.
\)
Thus for every \(k\),
\[
|T^{n_k}x-x|
<
e^{-S_{n_k}f(x)}.
\]
 Hence \(x\in R[T,f]\). Consequently,
\(x\in R[T,f]\cap\mathrm{Bad}(T,y)=R_{\mathrm{Bad}}[T,f,y].
\)
\end{proof}

\subsection{Mass distribution}
Now,
we  construct a probability measure supported on \( C_{\infty}\).
Let
\(r=M-\ell-2m_0.
\)
Only the free blocks carry independent choices. The safe blocks, connecting
blocks and returning blocks are prescribed once the free blocks have been
chosen.
For each \(u\in\Sigma^r\), let \(z_u\) be the midpoint of \(I(u)\).
By the
definition of \(t_r(f)\),
\[
\sum_{u\in\Sigma^r} e^{-t_r(f)S_rh(z_u)}=1.
\]
Hence, 
\(p(u):=e^{-t_r(f)S_rh(z_u)}
\)
defines a probability vector on \(\Sigma^r\).

For a word \( V_k\in\mathcal C_k\), let
\(u_i^{(j)}\in\Sigma^r,
\)
\((1\leq j\leq k\),  \(1\leq i\leq p_j)\),
be the free blocks appearing in \( V_k\). Define
\[
\mu\big(I_{n_k+m_k}( V_k)\big)
=
\prod_{j=1}^{k}\prod_{i=1}^{p_j}p(u_i^{(j)})
=
\prod_{j=1}^{k}\prod_{i=1}^{p_j}
e^{-t_r(f)S_rh(z_i^{(j)})},
\]
where \(z_i^{(j)}\) is the midpoint of \(I(u_i^{(j)})\).

Since the weights $p(u)$ sum to one at each free choice,
these assignments are consistent on the nested construction.
Equivalently, the mass of a child cylinder is obtained from
the mass of its parent by multiplying the weights of the
new free blocks.
Therefore, $\mu$ extends to a Borel probability measure.
Since $\max\{p(u): u\in\Sigma^r\}<1$ and each stage introduces
new free choices, $\mu$ is non-atomic, and hence
$\mu(C_\infty)=1$.

\subsection{Cylinder measure estimates and the dimension lower bound}
It remains to prove the following property: there exists \(C>0\) such
that every sufficiently small cylinder \(I_n(x)\) intersecting
\( C_{\infty}\) satisfies
\begin{equation}\label{eq-goal-new}
\mu(I_n(x))
\leq
C |I_n(x)|^{t(f)-3\delta-\frac{6\delta}{\log\lambda_2}}.
\end{equation}
Once this estimate is proved, Proposition~\ref{pro-mass} gives the desired
dimension lower bound.

We first estimate the measure of the cylinders $I_{n_k+m_k(x)}(x)$ where $x\in C_{\infty}$.
\begin{claim}\label{claim-basic-level}
For every $k\geq1$ and every $x\in C_\infty$,
\begin{equation}\label{case1-estimate-new}
\mu(I_{n_k+m_k}(x))
\leq
|I_{n_k+m_k}(x)|^{t(f)-2\delta-
\frac{6\delta}{\log\lambda_2}}.
\end{equation}
\end{claim}

\begin{proof}
Fix $x\in C_\infty$. For the remainder of the proof, write
$V_j=V_j(x)$, $W_j=W_j(x)$ and $m_j=m_j(x)$. For $j\geq2$,
also write $p_j=p_j(x)$ and $q_j=q_j(x)$, while $p_1$ is the
fixed integer in the definition of $\mathcal F_1$.

Let $c_0=\min\{f(x): x\in [0,1]\}$.
Since $r=M-\ell-2m_0\geq N_0>b$, Proposition~\ref{Pro1.4}
gives
\[
S_rh(z)
\geq
rc_0+(r-b)\log\lambda_2>0
\]
for every $z$ in an $r$-th order cylinder. Also,
\eqref{eq-tnf-new} gives
\(t_r(f)\geq t(f)-\delta.
\)
Moreover, Proposition~\ref{Pro1.4} and the disjointness of
cylinders of the same order imply
\[
\sum_{u\in\Sigma^n}e^{-S_nh(z_u)}
\leq
K_1e^{-nc_0}
\sum_{u\in\Sigma^n}|I(u)|
\leq
K_1e^{-nc_0}.
\]
Hence
\(P(T,-h)<0,
\)
and therefore $t(f)<1$. By the choice of $\delta$, we have
\[
0<t(f)-3\delta<1.
\]

By the definition of $\mu$,
\[
\mu(I_{n_k+m_k}(x))
\leq
\prod_{j=1}^{k}\prod_{i=1}^{p_j}
\left(
e^{-S_rf(z_i^{(j)})}
\frac{1}{|(T^r)'(z_i^{(j)})|}
\right)^{t(f)-\delta}.
\]
Let $z_j'$ be the midpoint of $I(W_j)$. 
Using \eqref{eq-S-nf-new}, $p_jM\leq n_j$, and the choices of
$M$ and $n_j$, we obtain, for every $j\geq1$,
\[
\left|
S_{n_j}f(z_j')
-
\sum_{i=1}^{p_j}S_rf(z_i^{(j)})
\right|
\leq
3\delta n_j.
\]
Consequently,
\begin{equation}\label{eq-measure-after-birkhoff-new}
\mu(I_{n_k+m_k}(x))
\leq
e^{3\delta\sum_{j=1}^{k}n_j}
\prod_{j=1}^{k}
\left(
e^{-S_{n_j}f(z_j')}
\prod_{i=1}^{p_j}
\frac{1}{|(T^r)'(z_i^{(j)})|}
\right)^{t(f)-\delta}.
\end{equation}

Repeated use of \eqref{eq-two-cylinder}, together with
Proposition~\ref{Pro1.4}, gives
\[
\prod_{i=1}^{p_1}
\frac{1}{|(T^r)'(z_i^{(1)})|}
\leq
c_1\lambda_1^{2M+m_0+b}
\left(
K_1c_1^2\lambda_1^{\ell+2m_0+b}
\right)^{p_1}
|I_{n_1}(x)|
\]
and, for every $j\geq2$,
\[
\begin{aligned}
|I_{n_{j-1}+m_{j-1}}(x)|
\prod_{i=1}^{p_j}
\frac{1}{|(T^r)'(z_i^{(j)})|}
\leq
c_1\lambda_1^{2M+m_0+b}
\left(
K_1c_1^2\lambda_1^{\ell+2m_0+b}
\right)^{p_j}
|I_{n_j}(x)|.
\end{aligned}
\]

On the other hand, by the minimality of $m_j$,
\eqref{eq-markov-m-k-lower-new}, and
Proposition~\ref{Pro1.4},
\[
\begin{aligned}
e^{-S_{n_j}f(z_j')}
\leq
e^{\delta n_j}
\left|
I_{m_j-1}
\bigl((W_j)^\infty|_1^{m_j-1}\bigr)
\right|
\leq
c_1\lambda_1^{b+1}e^{\delta n_j}
\left|
I_{m_j}
\bigl((W_j)^\infty|_1^{m_j}\bigr)
\right|.
\end{aligned}
\]
By \eqref{eq-two-cylinder}, 
\[
|I_{n_j}(x)|e^{-S_{n_j}f(z_j')}
\leq
c_1^2\lambda_1^{b+1}e^{\delta n_j}
|I_{n_j+m_j}(x)|.
\]

Combining the preceding estimates, we obtain
\[
\begin{aligned}
e^{-S_{n_1}f(z_1')}
\prod_{i=1}^{p_1}
\frac{1}{|(T^r)'(z_i^{(1)})|}
\leq
c_1^3\lambda_1^{2M+m_0+2b+1}
\left(
K_1c_1^2\lambda_1^{\ell+2m_0+b}
\right)^{p_1}
e^{\delta n_1}|I_{n_1+m_1}(x)|
\end{aligned}
\]
and, for every $j\geq2$,
\[
\begin{aligned}
&|I_{n_{j-1}+m_{j-1}}(x)|
\left(
e^{-S_{n_j}f(z_j')}
\prod_{i=1}^{p_j}
\frac{1}{|(T^r)'(z_i^{(j)})|}
\right)
\\
&\quad\leq
c_1^3\lambda_1^{2M+m_0+2b+1}
\left(
K_1c_1^2\lambda_1^{\ell+2m_0+b}
\right)^{p_j}
e^{\delta n_j}|I_{n_j+m_j}(x)|.
\end{aligned}
\]

Iterating these inequalities from $j=1$ to $j=k$, we have
\[
\begin{aligned}
&\prod_{j=1}^{k}
\left(
e^{-S_{n_j}f(z_j')}
\prod_{i=1}^{p_j}
\frac{1}{|(T^r)'(z_i^{(j)})|}
\right)
\\
&\quad\leq
c_1^{3k}\lambda_1^{k(2M+m_0+2b+1)}
\left(
K_1c_1^2\lambda_1^{\ell+2m_0+b}
\right)^{\sum_{j=1}^{k}p_j}
e^{\delta\sum_{j=1}^{k}n_j}
|I_{n_k+m_k}(x)|.
\end{aligned}
\]

Therefore,
\begin{equation}\label{eq-product-cylinder-new}
\begin{aligned}
\prod_{j=1}^{k}
\left(
e^{-S_{n_j}f(z_j')}
\prod_{i=1}^{p_j}
\frac{1}{|(T^r)'(z_i^{(j)})|}
\right)^{t(f)-\delta}
\leq
C_k e^{\delta\sum_{j=1}^{k}n_j}
|I_{n_k+m_k}(x)|^{t(f)-\delta},
\end{aligned}
\end{equation}
where
\[
C_k
=
c_1^{3k}\lambda_1^{k(2M+m_0+2b+1)}
\left(
K_1c_1^2\lambda_1^{\ell+2m_0+b}
\right)^{\sum_{j=1}^{k}p_j}.
\]
Here we used
\[
0<t(f)-\delta<1
\quad\text{and}\quad
C_k\geq1.
\]

By \eqref{eq-sparse-birkhoff-corrected},
\(\sum_{j=1}^{k}n_j\leq\frac32n_k.
\)
Combining \eqref{eq-measure-after-birkhoff-new} and
\eqref{eq-product-cylinder-new}, we obtain
\[
\begin{aligned}
\mu(I_{n_k+m_k}(x))
&\leq
C_k e^{4\delta\sum_{j=1}^{k}n_j}
|I_{n_k+m_k}(x)|^{t(f)-\delta}
\\
&\leq
C_k e^{6\delta n_k}
|I_{n_k+m_k}(x)|^{t(f)-\delta}.
\end{aligned}
\]

Moreover, Proposition~\ref{Pro1.4} gives
\[
|I_{n_k+m_k}(x)|
\leq
\lambda_2^{-(n_k+m_k)+b}.
\]
Hence,
\begin{equation}\label{eq-exp-error-new}
e^{6\delta n_k}
\leq
e^{6\delta b}
|I_{n_k+m_k}(x)|^{-\frac{6\delta}{\log\lambda_2}}.
\end{equation}
It follows that
\[
\mu(I_{n_k+m_k}(x))
\leq
C_ke^{6\delta b}
|I_{n_k+m_k}(x)|^{t(f)-\delta-
\frac{6\delta}{\log\lambda_2}}.
\]

Finally,
\[
n_k
=
M\sum_{j=1}^{k}p_j
+kM
+\sum_{j=1}^{k}q_j
+(k-1)m_0
+\sum_{j=1}^{k-1}m_j,
\]
and therefore
\(M\sum_{j=1}^{k}p_j\leq n_k.
\)
By \eqref{eq-M-distortion-new},
 and
\eqref{eq-sparse-last-new},
\[
\begin{aligned}
C_ke^{6\delta b}
&\leq
c_1^{3k}
\lambda_1^{k(2M+m_0+2b+1)}
e^{6\delta b}
\lambda_2^{\delta n_k/3}
\\
&\leq
\lambda_2^{\delta(n_k-b)}
\\
&\leq
|I_{n_k+m_k}(x)|^{-\delta}.
\end{aligned}
\]
Substituting this into the preceding estimate proves
\eqref{case1-estimate-new}.
\end{proof}

Now, we  pass from the  levels \(n_k+m_k\) to arbitrary  levels.
\begin{claim}\label{claim-intermediate-levels}
For every sufficiently large \(n\) and every cylinder \(I_n(x)\) intersecting
\( C_{\infty}\), the estimate \eqref{eq-goal-new} holds.
\end{claim}
\begin{proof}
Fix $x\in C_\infty$, and, only within this proof, put
\[
s=t(f)-2\delta-\frac{6\delta}{\log\lambda_2}>0.
\]
Choose \(k\) such that 
\( n_k+m_k\leq n<n_{k+1}+m_{k+1}. \)
We consider the three possible ranges for \(n\).

First,  suppose that
\[
n_k+m_k<n\leq n_k+m_k+M+m_0+q_{k+1}.
\]
Note that no new free block has been completed in this range. 
By \eqref{eq-two-cylinder} and
\eqref{6},
\[
\mu(I_n(x))=\mu(I_{n_k+m_k}(x))
\leq \left(c_1\lambda_1^{2M+m_0+b}\right)^s |I_n(x)|^s.
\]
This is stronger than \eqref{eq-goal-new}, up to a fixed constant.

Next suppose 
\(n_k+m_k+M+m_0+q_{k+1}<n\leq n_{k+1}\) and write
\[
n=n_k+m_k+M+m_0+q_{k+1}+jM+i,
\]
where \(0\leq j\leq p_{k+1}\) and \(0\leq i<M\). Ignoring any restrictions imposed by the remaining $i$ symbols,
we obtain
\[
\mu(I_n(x))\leq\mu(I_{n_k+m_k}(x))
\prod_{h=1}^{j}p(u_h^{(k+1)}),
\]
where $p(u)=e^{-t_r(f)S_rh(z_u)}$ and $r=M-\ell-2m_0$.
Since $t_r(f)\geq t(f)-\delta>s$, $f>0$, and
$|(T^r)'|>1$, we have
$p(u)\leq|(T^r)'(z_u)|^{-s}$.
 Repeated use of
\eqref{eq-two-cylinder} and the lower bound in \eqref{6} gives
\[
\begin{split}
|I_{n_k+m_k}(x)|
\prod_{h=1}^{j}\frac{1}{|(T^r)'(z_h^{(k+1)})|}
\leq c_1^2\lambda_1^{3M+m_0+2b}
\left(K_1c_1^2\lambda_1^{\ell+2m_0+b}\right)^j |I_n(x)|.
\end{split}
\]
Using Claim~\ref{claim-basic-level}, $0<s<1$ and
\eqref{eq-M-distortion-new}, we conclude that
\[
\begin{split}
\mu(I_n(x))
&\leq c_1^2\lambda_1^{3M+m_0+2b}
\lambda_2^{\delta jM/3}|I_n(x)|^s\\
&\leq c_1^2\lambda_1^{3M+m_0+2b}\lambda_2^{b\delta/3}
|I_n(x)|^{s-\delta/3}\\
&\leq c_1^2\lambda_1^{3M+m_0+2b}\lambda_2^{b\delta/3}
|I_n(x)|^{s-\delta}.
\end{split}
\]
Here $jM\leq n$ and $|I_n(x)|\leq\lambda_2^{-n+b}$ were used.
The constant is independent of $n$ and $k$.

Finally, if
\(n_{k+1}<n\leq n_{k+1}+m_{k+1}.
\)
Then, \(n\) lies inside the returning block at level \(k+1\). Hence
\[
\mu(I_n(x))=\mu(I_{n_{k+1}+m_{k+1}}(x))
\leq|I_{n_{k+1}+m_{k+1}}(x)|^s
\leq|I_n(x)|^s.
\]
This again implies \eqref{eq-goal-new}.
If $n=n_k+m_k$, the conclusion follows from
Claim~\ref{claim-basic-level}.
This proves the claim.
\end{proof}
By Claim~\ref{claim-intermediate-levels}, Proposition~\ref{pro-mass}, and
\( C_{\infty}\subset R_{\mathrm{Bad}}[T,f,y]\), we have
\[
\dim_{\mathrm H}R_{\mathrm{Bad}}[T,f,y]
\geq
\dim_{\mathrm H} C_{\infty}
\geq
t(f)-3\delta-\frac{6\delta}{\log\lambda_2}.
\]
Since $3\delta+6\delta/\log\lambda_2<\varepsilon$ and $\varepsilon>0$ was arbitrary, letting $\varepsilon\to0$ yields
\[
\dim_{\mathrm H}R_{\mathrm{Bad}}[T,f,y]\geq t(f),
\]
and the proof of Theorem~\ref{theorem-f-finite} is complete.
\section{Recurrence with \texorpdfstring{$e^{-S_n f}$}{exp(-S_n f)}: Countable Alphabets}
\label{proof-f-countable}

In this section, we prove Theorem~\ref{theorem-f-countable}.
Since
\(
R_{\mathrm{Bad}}[T,f,y]\subset R[T,f],
\)
the upper bound follows from the known dimension formula for \(R[T,f]\) in \cite{Yuan2023}.
It
remains to prove
\(\dim_{\mathrm H}R_{\mathrm{Bad}}[T,f,y]\geq t(f),
\)
where
\(
t(f)
=
\inf\left\{
t\geq0:
P(T,-t(f+\log|T'|))\leq0
\right\}.
\)

Assume that $\Lambda=\mathbb N$.
Fix $N\geq2$ and put $A_N=\{1,\ldots,N\}$. Proposition~\ref{6.5}
gives
\begin{equation}\label{eq-countable-finite-pressure}
\lim_{N\to\infty}t(A_N,f)=t(f).
\end{equation}
For this fixed $N$, the finite restriction of $T$ satisfies
Definition~\ref{1.1} with  $K$ and $m_0=0$, as
explained in Section~\ref{pressure-function}. 
The
constants $\lambda_1,\lambda_2,b$ are fixed before the construction
parameters are chosen. For $\varepsilon>0$, choose
$\delta>0$ so that
\[
3\delta+6\delta/\log\lambda_2<\min\{\varepsilon,t(A_N,f)\}.
\]
Then choose $N_0,M,\ell$ and the return times in the order specified
in Section~\ref{proof-main-theorem-finite}, using free blocks over
$A_N$.

The full image
property allows these words to be concatenated directly, so no connecting
block is inserted.
For these values of $M$ and $\ell$, define $\mathcal U(M,\ell)$ by the formula in
Section~\ref{proof-main-theorem-finite}. 
Since every word over $A_N$ is
admissible and
\(\#A_N^{\ell}=N^{\ell},
\)
the parameters can be chosen so that
\(\#A_N^{\ell}>\#\mathcal U(M,\ell).
\)
Choose
\[
    v\in A_N^{\ell}\setminus\mathcal U(M,\ell).
\]
The target $y$ need not belong to $E_{A_N}$. For avoidance, use
$\mathcal U(M,\ell)$ from the original system $(X, T)$. 
Thus the proof only requires $y\in E$.
Repeat the construction in the proof of Theorem~\ref{theorem-f-finite} with
all free blocks over $A_N$, no connecting block, and copies of the word
$v\in A_N^{\ell}\setminus\mathcal U(M,\ell)$. Let $\mathcal C_{A_N}$ denote the
Cantor set. The
repeated initial words give
\[
    |T^{n_k}x-x|<e^{-S_{n_k}f(x)}.
\]
The copies of
$v\notin\mathcal U(M,\ell)$ give, by the argument of
Lemma~\ref{lemma-subset-new}, a positive lower bound for $|T^n x-y|$ for all
sufficiently large $n$. Hence,
\[
    \mathcal C_{A_N}\subset R_{\mathrm{Bad}}[T,f,y].
\]
For fixed $N$, every cylinder determined by a word over $A_N$
is also a cylinder of the original system. Hence the bounded distortion
and cylinder estimates used in the proof of
Theorem~\ref{theorem-f-finite} remain valid, and the same mass
distribution argument applies.
 The choice of $\delta$ gives
\[
\dim_{\mathrm H}\mathcal C_{A_N}\geq t(A_N,f)-\varepsilon.
\]
The construction depends on $\varepsilon$. Letting $\varepsilon\downarrow0$
for this fixed $N$ yields
\[
\dim_{\mathrm H}R_{\mathrm{Bad}}[T,f,y]\geq t(A_N,f).
\]
Now let $N\to\infty$ in \eqref{eq-countable-finite-pressure}.
This gives the desired lower bound.

\section{Recurrence with \texorpdfstring{$\psi$}{psi}: Finite Alphabets}
\label{therorem1.3-proof}

In this section, we prove Theorem~\ref{main-theorem-finite}. Since
\(R_{\mathrm{Bad}}(T,\psi,y)\subset R(T,\psi),
\)
the upper bound follows from the known dimension formula for \(R(T,\psi)\) in \cite{Yuan2023}.
It remains to prove the lower bound.

Recall that
\[
\log\gamma
=
\liminf_{n\to\infty}\frac{-\log\psi(n)}{n},
\qquad
\gamma\in[1,\infty].
\]
First, assume that \(1<\gamma<\infty\) and fix \(\gamma_1>\gamma\). 
 By the
definition of $\gamma$, there are infinitely many $n$ such that
\(\psi(n)>\gamma_1^{-n}.
\)
Choose $\{n_k\}_{k\geq1}$ from the integers satisfying
$\psi(n)>\gamma_1^{-n}$ and impose the sparseness conditions used in the proof
of Theorem~\ref{theorem-f-finite}.
Apply that construction with the constant function
$f\equiv\log\gamma_1$.
At every selected integer $n_k$,
\[
|T^{n_k}x-x|<\gamma_1^{-n_k}<\psi(n_k),
\]
and the safe blocks give \(x\in\mathrm{Bad}(T,y)\).
Hence, the mass distribution
argument from Theorem~\ref{theorem-f-finite}, with
\(f\equiv\log\gamma_1\), yields
\[
\dim_{\mathrm H}R_{\mathrm{Bad}}(T,\psi,y)\geq t(\Lambda,\gamma_1),
\]
where $t(\Lambda,\gamma_1)$ is the unique root of
\(P\bigl(T,-t(\log\gamma_1+\log|T'|)\bigr)=0.
\)
Letting \(\gamma_1\downarrow\gamma\) and using Proposition~\ref{0.4} give
\[
\dim_{\mathrm H}R_{\mathrm{Bad}}(T,\psi,y)\geq t(\Lambda,\gamma).
\]

Now suppose that $\gamma=1$. Fix $\alpha>1$ and define
\[
\overline{\psi}(n)=\min\{\psi(n),\alpha^{-n}\}.
\]
We have $\overline\psi\leq\psi$ and
\[
\liminf_{n\to\infty}
\frac{-\log\overline\psi(n)}{n}
=
\log\alpha.
\]
Applying the case $1<\gamma<\infty$ to $\overline\psi$ gives
\[
\dim_{\mathrm H}R_{\mathrm{Bad}}(T,\psi,y)
\geq
\dim_{\mathrm H}R_{\mathrm{Bad}}(T,\overline\psi,y)
\geq
t(\Lambda,\alpha).
\]
Letting \(\alpha\downarrow1\) gives the desired lower bound
\(\dim_{\mathrm H}R_{\mathrm{Bad}}(T,\psi,y)\geq t(\Lambda,1).
\)

Now, assume  $\gamma=\infty$.
By the dimension formula for
\(R(T,\psi)\) proved in \cite{Yuan2023},
\(\dim_{\mathrm H}R(T,\psi)=0.
\)
Since
\(R_{\mathrm{Bad}}(T,\psi,y)\subset R(T,\psi)\), we have
\(\dim_{\mathrm H}R_{\mathrm{Bad}}(T,\psi,y)=0
\).
This completes the proof of Theorem~\ref{main-theorem-finite}.

\section{Recurrence with \texorpdfstring{$\psi$}{psi}: Countable Alphabets}
\label{proof-psi-countable}
In this section,
we prove Theorem~\ref{theorem-countable}. Since
\(R_{\mathrm{Bad}}(T,\psi,y)\subset R(T,\psi),
\)
the upper bound follows from the known dimension formula for \(R(T,\psi)\) in \cite{Yuan2023}.
It remains to prove the lower bound.

Assume that $\Lambda=\mathbb N$ and $1\leq\gamma<\infty$.
Fix $\gamma_1>\gamma$ and $N\geq2$, and put $A_N=\{1,\ldots,N\}$.
Use the map $T$ on these finitely many branches, with
 $K$ and with $m_0=0$. 
 After the
constants $\lambda_1,\lambda_2,b$ are fixed, apply Section~\ref{proof-main-theorem-finite} with
$f\equiv\log\gamma_1$ and free words over $A_N$.
For $\varepsilon>0$, choose the construction parameter
$\delta$ so that
\[
3\delta+6\delta/\log\lambda_2
<\min\{\varepsilon,t(A_N,\gamma_1)\},
\]
and then choose $N_0,M,\ell$ and the remaining fixed parameters.
At each stage choose $n_k$ from
$\{n:\psi(n)>\gamma_1^{-n}\}$, subject to the sparseness
conditions and the lengths already constructed.

For these values of $M$ and $\ell$, choose
\[
    v\in A_N^{\ell}\setminus\mathcal U(M,\ell),
\]
which is possible after the parameters are chosen so that
$N^{\ell}>\#\mathcal U(M,\ell)$.

Let $\mathcal C_{A_N}$ denote the Cantor set obtained with free blocks over
$A_N$, no connecting blocks, and copies of
$v\in A_N^{\ell}\setminus\mathcal U(M,\ell)$. At the return times,
\[
    |T^{n_k}x-x|<\gamma_1^{-n_k}<\psi(n_k).
\]
The copies of $v\notin\mathcal U(M,\ell)$ give a positive lower bound for the
distance from the orbit to $y$. Hence,
\(
\mathcal C_{A_N}
\subset
R_{\mathrm{Bad}}(T,\psi,y).
\)
Every word used in the construction has symbols in $A_N$.  Using the partition sums over $A_N$ in the measure
construction and repeating the calculation in
\eqref{eq-measure-after-birkhoff-new}--\eqref{eq-product-cylinder-new} gives
\[
\dim_{\mathrm H}\mathcal C_{A_N}\geq t(A_N,\gamma_1)-\varepsilon.
\]
First let $\varepsilon\downarrow0$ for fixed $N,\gamma_1$, and then
let $N\to\infty$, using Proposition~\ref{6.1} with the positive
constant function $f=\log\gamma_1$. This gives
\[
\dim_{\mathrm H}R_{\mathrm{Bad}}(T,\psi,y)\geq t(\Lambda,\gamma_1).
\]
Finally let $\gamma_1\downarrow\gamma$ and use
Proposition~\ref{0.4}. This proves Theorem~\ref{theorem-countable}.

\section{Intersections between approximation and avoidance sets}
\label{proof-theorem-re-sh-f}

In this section, we prove
Theorem~\ref{theorem-recurrent-shrinking-psi}.

If $\psi(n)\to0$, every point in $R(T,\psi)$ satisfies
\(\liminf_{n\to\infty}|T^n x-x|=0,
\)
and every point in $S(T,\psi,y)$ satisfies
\(\liminf_{n\to\infty}|T^n x-y|=0.
\)
Both intersections are therefore empty by the definitions of $N(T)$
and $\mathrm{Bad}(T,y)$.

It remains to investigate the case
\[
c:=\limsup_{n\to\infty}\psi(n)>0.
\]
Both intersections are contained in $E$, so their Hausdorff dimensions are at
most $\dim_{\mathrm H}E$. We prove the reverse inequalities.

The return times below are chosen successively from the infinite
set $\{n\in\mathbb N:\psi(n)>c/4\}$.

\subsection{Finite alphabets}

Assume that $\Lambda=\{1,\ldots,d\}$ is finite.
If $\dim_{\mathrm H}E=0$, the conclusion is immediate.
Otherwise, fix $0<\varepsilon<\dim_{\mathrm H}E$ and put
\[
s=\dim_{\mathrm H}E-\varepsilon/2.
\]

We construct two Cantor sets $F_\infty$ and $G_\infty$,
corresponding to $R(T,\psi)\cap N(T)$ and
$S(T,\psi,y)\cap\mathrm{Bad}(T,y)$, respectively,
following the construction in
Section~\ref{proof-main-theorem-finite}.

For $R(T,\psi)\cap N(T)$, let $j_1$ be the first symbol
of $\bar{\mathtt u}$ when $m_0>0$, and choose any
$j_1\in\Lambda$ when $m_0=0$.
By Proposition~\ref{Pro1.4}, $(j_1)^n$ is admissible
for every $n\geq1$.
Choose $t$ such that
\[
|I((j_1)^{t+1})|<c/4
\]
and put $w=(j_1)^{t+1}$.

For $S(T,\psi,y)\cap\mathrm{Bad}(T,y)$, choose $t$ such that
$|I_t(y)|<c/4$ and put $w=y|_1^t$.
In each case, choose $M,\ell$ and
$r=M-\ell-2m_0$ as in
Section~\ref{proof-main-theorem-finite}, with
\[
\begin{gathered}
M>|w|+\ell+3m_0,\qquad r>b,\qquad t_r(1)\geq s,\qquad
\frac{\log(K_1c_1^2\lambda_1^{\ell+2m_0+b})}{M}
<\frac{\varepsilon\log\lambda_2}{4}.
\end{gathered}
\]

Put $L=3M+\ell-1$. For the construction of $F_\infty$, define
\[
\mathcal U_F(M,\ell)
=
\{(j_1)^\ell\}
\cup
\left\{
\left(\big((j_1)^L\big)^+\right)|_i^{i+\ell-1},
\left(\big((j_1)^L\big)^-\right)|_i^{i+\ell-1}
:1\leq i\leq3M
\right\},
\]
where terms corresponding to nonexistent adjacent
cylinders are omitted.
Since
\[
\#\mathcal U_F(M,\ell)\leq6M+1
<\#\Sigma^\ell,
\]
we choose
\(v\in\Sigma^\ell\setminus\mathcal U_F(M,\ell).
\)

For the construction of $G_\infty$, choose
$v\in\Sigma^\ell\setminus\mathcal U(M,\ell)$
as in Section~\ref{proof-main-theorem-finite}.

Set $\bar{\mathtt v}=(\bar{\mathtt u},v,\bar{\mathtt u})$
and choose the initial word
\[
V_0=
\begin{cases}
((j_1)^{3M+\ell-1},\bar{\mathtt v}),
&\text{for the construction of }F_\infty,\\
\bar{\mathtt v},
&\text{for the construction of }G_\infty.
\end{cases}
\]
As in Section~\ref{proof-main-theorem-finite}, choose
$n_k$ recursively from the integers satisfying
$\psi(n_k)>c/4$, sufficiently sparse that
\[
n_k\geq2\sum_{j<k}n_j,\qquad
n_k-|V_{k-1}|-m_0=p_kM+q_k,
\]
where $p_k\geq1$ and $0\leq q_k<M$.
Define
\[
\begin{aligned}
W_k=(V_{k-1},u_1^{(k)},\bar{\mathtt v},\ldots,
u_{p_k}^{(k)},\bar{\mathtt v},
(\bar{\mathtt u}^{\infty})|_1^{q_k},\bar{\mathtt u}),\qquad
V_k=(W_k,w,\bar{\mathtt v}),
\end{aligned}
\]
where $u_i^{(k)}\in\Sigma^r$ are chosen independently.
Then
\(|W_k|=n_k,
\)
so $w$ begins exactly at position $n_k+1$.

The starting positions of the occurrences of $v$
inserted after consecutive free blocks differ by $M$.
Across an inserted word $w$, the corresponding distance is
\(q_k+|w|+\ell+3m_0<2M.
\)
Consequently, every sufficiently late block of length
$3M+\ell-1$ contains an occurrence of $v$.

Let $F_\infty$ and $G_\infty$ be the resulting sets obtained
from the corresponding nested cylinders. The choice of $w$ and
Lemma~\ref{lemma-returning-block} give
\[
F_\infty\subset R(T,\psi),
\qquad
G_\infty\subset S(T,\psi,y).
\]
For every $x\in F_\infty$, the initial word is
$(j_1)^L$. Since every sufficiently late block
of length $L$ contains $v\notin\mathcal U_F(M,\ell)$,
the word $x|_{n+1}^{n+L}$ cannot be $(j_1)^L$
or either of its adjacent words of length $L$,
for all sufficiently large $n$.

Since $\Lambda$ is finite, the closures of the
remaining $L$-th order cylinders have positive
distance from $\overline{I((j_1)^L)}$.
Consequently, there exists $a>0$ such that
\[
|T^nx-x|\geq a
\]
for every $x\in F_\infty$ and all sufficiently
large $n$. Hence $F_\infty\subset N(T)$.

For every $x\in G_\infty$, the choice
$v\notin\mathcal U(M,\ell)$ implies that, for all
sufficiently large $n$, the word
$x|_{n+1}^{n+L}$ differs from $y|_1^L$ and
its adjacent words $(y|_1^L)^+$ and $(y|_1^L)^-$,
whenever they exist. Indeed, otherwise the occurrence
of $v$ among these $L$ symbols would belong to
$\mathcal U(M,\ell)$, a contradiction.
Hence $T^nx$ belongs neither to $I_L(y)$ nor to
either of its adjacent $L$-th order cylinders.
Since $y\in E$, the cylinder $I_L(y)$ is an open
interval containing $y$. 
There exists $a_G>0$ such that 
\[
|T^nx-y|\geq a_G
\]
for every $x\in G_\infty$ and all sufficiently
large $n$. Consequently,
$G_\infty\subset\mathrm{Bad}(T,y)$.
Thus
\[
F_\infty\subset R(T,\psi)\cap N(T),
\qquad
G_\infty\subset
S(T,\psi,y)\cap\mathrm{Bad}(T,y).
\]

It remains to estimate their dimensions.
For both constructions, assign to each free word
$u\in\Sigma^r$ the weight
\[
p(u)=|(T^r)'(z_u)|^{-t_r(1)},
\qquad
\sum_{u\in\Sigma^r}p(u)=1.
\]

Here $z_u$ is the midpoint of $I(u)$.
The weights $p(u)$ define a probability measure $\mu$
as in Section~\ref{proof-main-theorem-finite}.

Each word $(u_i^{(k)},\bar{\mathtt v})$ has length $M$.
The remaining word inserted at stage $k$ is
\(((\bar{\mathtt u}^{\infty})|_1^{q_k},
\bar{\mathtt u},w,\bar{\mathtt v}),
\)
whose length is
$q_k+|w|+\ell+3m_0<M+|w|+\ell+3m_0$.
Since $n_k\geq2\sum_{j<k}n_j$, the number of such
words among the first $n$ symbols is $O(\log n)$.
Any incomplete word at position $n$ also has
uniformly bounded length.
Using \eqref{eq-two-cylinder} and \eqref{6} as in
Section~\ref{proof-main-theorem-finite}, we obtain
\[
\mu(I_n(x))
\leq
C\exp\left(
\frac{n}{M}
\log(K_1c_1^2\lambda_1^{\ell+2m_0+b})
+O(\log n)
\right)|I_n(x)|^s.
\]
Here the constants are independent of $n$ and $x$.
By the choice of $M$ and
$|I_n(x)|\leq\lambda_2^{-n+b}$, for all sufficiently
large $n$,
\[
\mu(I_n(x))
\leq C_\varepsilon
|I_n(x)|^{s-\varepsilon/2}
=
C_\varepsilon
|I_n(x)|^{\dim_{\mathrm H}E-\varepsilon}.
\]
Proposition~\ref{pro-mass} therefore gives
\[
\dim_{\mathrm H}F_\infty,\;
\dim_{\mathrm H}G_\infty
\geq\dim_{\mathrm H}E-\varepsilon.
\]
Letting $\varepsilon\downarrow0$ proves the required
lower bounds in the finite-alphabet case.

\subsection{Countable alphabets}

Assume that $\Lambda$ is countable. By the full image
property, admissible words can be concatenated directly.
Let $A_N=\{1,\ldots,N\}\subset\Lambda$, and denote by
$E_{A_N}$ the set of points in $E$ whose symbolic
representations belong to $A_N^{\mathbb N}$.

First, fix the prescribed word $w$ for each of the two
constructions as above, taking $j_1=1$ in the recurrence
case. Since these words are finite, choose $N$ sufficiently
large that all their symbols belong to $A_N$.

For this fixed $N$, the finite-alphabet construction
applies to the map $T$, with the same $K$ and
$m_0=0$. For each
$0<\varepsilon<\dim_{\mathrm H}E_{A_N}$, choose the
remaining parameters and the safe word
$v\in A_N^\ell$ as in the finite-alphabet case.
For shrinking targets, $\mathcal U(M,\ell)$ is defined
using the target $y$, which need not belong
to $E_{A_N}$.

Let $F_N$ and $G_N$ denote the resulting Cantor sets.
The inclusion arguments and the cylinder measure estimates
proved in the finite-alphabet case give
\[
\begin{gathered}
F_N\subset R(T,\psi)\cap N(T),\qquad
G_N\subset S(T,\psi,y)\cap\mathrm{Bad}(T,y),\\
\dim_{\mathrm H}F_N\geq\dim_{\mathrm H}E_{A_N}-\varepsilon,
\qquad
\dim_{\mathrm H}G_N\geq\dim_{\mathrm H}E_{A_N}-\varepsilon.
\end{gathered}
\]
The dependence of $F_N,G_N$ on $\varepsilon$ is suppressed
in the notation. First let $\varepsilon\downarrow0$ with
$N$ fixed, and then let $N\to\infty$.
Since
\(\dim_{\mathrm H}E_{A_N}\longrightarrow\dim_{\mathrm H}E
\)
by the finite-subsystem approximation in
Section~\ref{pressure-function}, both lower bounds follow.
Together with the upper bounds, this completes the proof
of Theorem~\ref{theorem-recurrent-shrinking-psi}.

\section*{Acknowledgments}
The first author is partially supported by Beijing Natural Science Foundation (JQ26002).
The second author is partially supported by  Guangdong University
Young Innovative Talents Program Project (No. 2025KQNCX059).

\bibliographystyle{amsplain}
\bibliography{ref}

@misc{He2025MTP,
  author        = {He, Y.-B.},
  title         = {A Dimensional Mass Transference Principle from Balls
                   to Open Sets and Applications to Dynamical
                   Diophantine Approximation},
  year          = {2025},
  note= {Preprint: arXiv: 2508.03359}
}

@article{He2023QuantitativeRP,
  title={Quantitative recurrence properties for piecewise expanding maps on $[0,1]^d$},
  author={Y.-B. He and L.-M. Liao},
  journal={Annali Scuola Normale Superiore-Classe di Scienze},
  year={2026},
  volume ={27},
   number = {2},
   pages={969-1014}
}

@article{Yuan2023,
  doi = {10.1080/14689367.2023.2184328},
  year = {2023},
  month = feb,
  publisher = {Informa {UK} Limited},
  author = {N. Yuan and B. Li},
  title = {Hausdorff dimensions of recurrent and shrinking target sets under {Lipschitz} functions for expanding {Markov} maps},
 journal = {Dyn. Syst.},
 pages = {365–394},
  number = {3},
  volume = {38}
}

@article{yuan2022,
author = {N. Yuan and B. Li and M. Wu},
title = {BADLY APPROXIMABLE AND NONRECURRENT SETS FOR EXPANDING MARKOV MAPS},
journal = {Fractals},
volume = {29},
number = {06},
pages = {2150170},
year = {2021},
doi = {10.1142/S0218348X2150170X},
URL = {   https://doi.org/10.1142/S0218348X2150170X}
}

@book{Falconer1990,
  title={Fractal Geometry-Mathematical Foundations and Application},
  author={K. J. Falconer},
  publisher ={Wiley, New York},
  year={1990}
}

@book{mane1990,
  title={Ergodic Theory and Differentiable Dynamics},
  author={R. {Ma\~{n}\'{e}}},
  publisher ={Springer-Verlag},
  year={1987}
}

@article{Hill1995,
  doi = {10.1007/bf01245179},
  year = {1995},
  month = dec,
  publisher = {Springer Science and Business Media {LLC}},
  volume = {119},
  number = {1},
  pages = {175--198},
  author = {R. Hill and S. L. Velani},
  title = {The ergodic theory of shrinking targets},
  journal = {Invent. Math.}
}

@article{Baker2019,
    title={{Quantitative recurrence properties for self-conformal sets.}},
    year={2021},
    journal = {Proc. Amer. Math. Soc.},
    author = {S. Baker and M. Farmer},
    number = {3},
     month = {3},
    pages = {1127–1138},
    volume = {149},
    doi = {10.1090/proc/15285},
    issn = {1088-6826}
}

@article{Chang2018,
    author = {Y.-Y. Chang and M. Wu and W. Wu},
    journal = {Proc. Amer. Math. Soc.},
    number = {4},
    year = {2019},
    title = {{Quantitative recurrence properties and homogeneous self-similar sets}},
    volume = {147},
    pages = {1453--1465}
}

@article{Hussain2021,
    author = {Hussain, M. and Li, B. and Simmons, D. and Wang, B.-W.},
    journal = {Ergodic Theory Dynam. Systems},
      volume={42},
 number = {6},
    pages = {1994-2008},
    title = {{Dynamical Borel-Cantelli lemma for recurrence theory}},
    year = {2022}
}

@article{koivusalo2018, 
title={{Recurrence to shrinking targets on typical self-affine fractals}}, 
volume={61}, 
number={2}, 
journal={Proc. Edinb. Math. Soc.}, 
publisher={Cambridge University Press}, 
author={Koivusalo, H. and Ramírez, F. A.}, 
year={2018},
pages={387-400}
}

@article{li2022,
  year = {2023},
  publisher = {Elsevier {BV}},
  volume = {421},
  author = {B. Li and L.-M. Liao and S. Velani and E. Zorin},
  title = {The shrinking target problem for matrix transformations of tori: Revisiting the standard problem},
  journal = {Adv. Math. },
pages={108994}
}

@article{Li2014,
  title={The shrinking target problem in the dynamical system of continued fractions},
  author={B. Li and B.-W. Wang and J. Wu and J. Xu},
  journal={Proc. Lond. Math. Soc.},
  year={2014},
  volume={108},
    number = {1},
   pages={159-186} 
}

@article{Shen2013,
  title={Shrinking target problems for beta-dynamical system},
  author={L.-M. Shen and B.-W. Wang},
  journal={Sci. China Math.},
  year={2013},
  volume={56},
  number={1},
  pages={91-104}
}

@article{Seuret2013,
  title={{Quantitative recurrence properties in conformal iterated function systems}},
  author={S. Seuret and B.-W. Wang},
  journal={Adv. Math.},
  year={2015},
  volume={280},
  pages={472-505}
}

@article{Tan2011,
  title={{Quantitative recurrence properties for beta-dynamical system}},
  author={B. Tan and B.-W. Wang},
  journal={Adv. Math.},
  year={2011},
  volume={228},
    number={4},
  pages={2071-2097}
}

@article{Jarnik1931,
author = {V. Jarn\'{i}k},
journal = {Math. Z.},
pages = {505-543},
title = {\"{U}ber die simultanen diophantischen Approximationen},
volume = {33},
year = {1931},
}

@article{Khintchine1924,
author = {Khintchine, A.},
journal = {Math. Ann.},
pages = {115-125},
title = {Einige S\"{a}tze \"{u}ber Kettenbr\"{u}che, mit Anwendungen auf die Theorie der Diophantischen Approximationen},
volume = {92},
year = {1924},
}

@article{Hill1997,
  year = {1997},
  volume = {85},
  number = {1},
  pages = {193--216},
  author = {R. Hill and S. L. Velani},
  title = {{Metric {Diophantine} approximation in Julia sets of expanding rational maps}},
  journal = {Publ. Math. Inst. Hautes  \'{E}tudes Sci.}
}

@article{barany2018,
  title = {Shrinking targets on {Bedford-McMullen} carpets},
  volume = {117},
  ISSN = {0024-6115},
  url = {http://dx.doi.org/10.1112/plms.12151},
  DOI = {10.1112/plms.12151},
  number = {5},
  journal = { Proc. Lond. Math. Soc.},
  publisher = {Wiley},
  author = {B. B\'ar\'any and M. Rams },
  year = {2018},
  month = jun,
  pages = {951–995}
}

@article{Beresnevich2006,
  title = {Measure theoretic laws for lim sup sets},
  volume = {179},
  ISSN = {1947-6221},
  url = {http://dx.doi.org/10.1090/memo/0846},
  DOI = {10.1090/memo/0846},
  number = {846},
  journal = { Mem. Amer. Math. Soc. },
  publisher = {American Mathematical Society (AMS)},
  author = {V.  Beresnevich and D. Dickinson and S. Velani},
  year = {2006},
  pages = { x+91 pp}
}

@article {Levesley2007,
    AUTHOR = {J. Levesley and C. Salp and S. L. Velani},
     TITLE = {On a problem of {K}. {M}ahler: {D}iophantine approximation and
              {C}antor sets},
   JOURNAL = {Math. Ann.},
  FJOURNAL = {Mathematische Annalen},
    VOLUME = {338},
      YEAR = {2007},
    NUMBER = {1},
     PAGES = {97--118},
      ISSN = {0025-5831,1432-1807},
   MRCLASS = {11J83 (11J82 11K55 37A45)},
MRREVIEWER = {R.\ Wallisser},
       DOI = {10.1007/s00208-006-0069-8},
       URL = {https://doi.org/10.1007/s00208-006-0069-8},
}

@article {Tanwang2024,
    AUTHOR = {B. Tan and B.-W. Wang and J. Wu},
     TITLE = {Mahler's question for intrinsic {D}iophantine approximation on
              triadic {C}antor set: the divergence theory},
   JOURNAL = {Math. Z.},
  FJOURNAL = {Math. Z.},
    VOLUME = {306},
      YEAR = {2024},
    NUMBER = {1},
     PAGES = {Paper No. 2, 24},
      ISSN = {0025-5874,1432-1823},
   MRCLASS = {11J83 (11K55)},
MRREVIEWER = {Gerardo\ Gonz\'alez Robert},
       DOI = {10.1007/s00209-023-03397-1},
       URL = {https://doi.org/10.1007/s00209-023-03397-1},
}

@misc{Heweikun2024,
title={Khintchine dichotomy for self-similar measures},
author={T. B\'{e}nard and W.-K. He  and H. Zhang },
note={J. Amer.
Math. Soc., 2025. Published online. https://doi.org/10.1090/jams/1070.
},
}

@misc{Heweikun2025,
title={Khintchine dichotomy and Schmidt estimates for self-similar
measures on $\mathbb{R}^n$},
author={T. B\'{e}nard and W.-K. He  and H. Zhang},
year={2025},
note={Preprint: arXiv: 2508.09076},
}

@article {Bugeaud2016,
    AUTHOR = {Y. Bugeaud and  A. Durand},
     TITLE = {Metric {D}iophantine approximation on the middle-third
              {C}antor set},
   JOURNAL = {J. Eur. Math. Soc. (JEMS)},
  FJOURNAL = {Journal of the European Mathematical Society (JEMS)},
    VOLUME = {18},
      YEAR = {2016},
    NUMBER = {6},
     PAGES = {1233--1272},
      ISSN = {1435-9855,1435-9863},
   MRCLASS = {11J83 (11J82 28A80 60D05)},
MRREVIEWER = {Dzmitry\ Badziahin},
       DOI = {10.4171/JEMS/612},
       URL = {https://doi.org/10.4171/JEMS/612},
}

@article {Khalil36,
    AUTHOR = {Khalil, O. and Luethi, M.},
     TITLE = {Random walks, spectral gaps, and {K}hintchine's theorem on
              fractals},
   JOURNAL = {Invent. Math.},
  FJOURNAL = {Inventiones Mathematicae},
    VOLUME = {232},
      YEAR = {2023},
    NUMBER = {2},
     PAGES = {713--831},
      ISSN = {0020-9910,1432-1297},
   MRCLASS = {11J83 (28A80)},
MRREVIEWER = {Johannes\ Schleischitz},
       DOI = {10.1007/s00222-022-01171-4},
       URL = {https://doi.org/10.1007/s00222-022-01171-4},
}

@article{Mahler1984,
    author ={K. Mahler},
    title = {Some suggestions for further research},
    journal = {Bull. Austral. Math. Soc.},
    year = {1984},
volum={29},
pages={101–108}
}

@book{Schmidt1980,
    author = {W. M. Schmidt},
    title = {{Diophantine approximation}},
    publisher = {volume 785 of Lecture Notes in Mathematics. Springer,
Berlin},
    year = {1980}
}

@article {Broderick2011,
    AUTHOR = {Broderick, R. and Fishman, L. and Kleinbock, D.},
     TITLE = {Schmidt's game, fractals, and orbits of toral endomorphisms},
   JOURNAL = {Ergodic Theory Dynam. Systems},
  FJOURNAL = {Ergodic Theory and Dynamical Systems},
    VOLUME = {31},
      YEAR = {2011},
    NUMBER = {4},
     PAGES = {1095--1107},
      ISSN = {0143-3857,1469-4417},
   MRCLASS = {11J83 (11K55 28A80)},
MRREVIEWER = {Thomas\ Ward},
       DOI = {10.1017/S0143385710000374},
       URL = {https://doi.org/10.1017/S0143385710000374},
}

@article {Das2018,
    AUTHOR = {Das, T. and Fishman, L. and Simmons, D. and
              Urba\'nski, M.},
     TITLE = {Badly approximable vectors and fractals defined by conformal
              dynamical systems},
   JOURNAL = {Math. Res. Lett.},
  FJOURNAL = {Mathematical Research Letters},
    VOLUME = {25},
      YEAR = {2018},
    NUMBER = {2},
     PAGES = {437--467},
      ISSN = {1073-2780,1945-001X},
   MRCLASS = {37F35 (28A80 37C45 37F10)},
MRREVIEWER = {Manuel\ Stadlbauer},
       DOI = {10.4310/MRL.2018.v25.n2.a5},
       URL = {https://doi.org/10.4310/MRL.2018.v25.n2.a5},
}

@article {Das2019,
    AUTHOR = {Das, T. and Fishman, L. and Simmons, D. and
              Urba\'nski, M.},
     TITLE = {Badly approximable points on self-affine sponges and the lower
              {A}ssouad dimension},
   JOURNAL = {Ergodic Theory Dynam. Systems},
  FJOURNAL = {Ergodic Theory and Dynamical Systems},
    VOLUME = {39},
      YEAR = {2019},
    NUMBER = {3},
     PAGES = {638--657},
      ISSN = {0143-3857,1469-4417},
   MRCLASS = {28A80 (11K55 37C45)},
MRREVIEWER = {J\"org\ Neunh\"auserer},
       DOI = {10.1017/etds.2017.42},
       URL = {https://doi.org/10.1017/etds.2017.42},
}

@article{Schmidt1966,
    author = { W. M. Schmidt},
    title = {  On badly approximable numbers and certain games},
    journal ={Trans. Amer. Math. Soc.},
    year ={ 1966},
    pages={ 178–199},
    volume={123},
}

@article {Hu2014,
    AUTHOR = {Hu, H. and Yu, Y.-L.},
     TITLE = {On {S}chmidt's game and the set of points with non-dense
              orbits under a class of expanding maps},
   JOURNAL = {J. Math. Anal. Appl.},
  FJOURNAL = {Journal of Mathematical Analysis and Applications},
    VOLUME = {418},
      YEAR = {2014},
    NUMBER = {2},
     PAGES = {906--920},
      ISSN = {0022-247X,1096-0813},
   MRCLASS = {37F35},
MRREVIEWER = {David\ Jonathon\ Sixsmith},
       DOI = {10.1016/j.jmaa.2014.04.026},
       URL = {https://doi.org/10.1016/j.jmaa.2014.04.026},
}

@article{Hu2025,
    author = {H. Hu and Z.-H. Li and Y.-L. Yu},
    title = {Schmidt's game and nondense orbits of a class of interval maps},
    journal = { Dyn. Syst.},
    year = {2025},
    pages={1-14},
}

@misc{huang2025,
      title={A winning approach to the intersections of twisted non-recurrent sets with fractals}, 
      author={J.-J. Huang and B. Li and B. Wang and N. Yuan},
      year={2025},
      note={Preprint: arXiv: 2512.07686},
      primaryClass={math.DS},
      url={https://arxiv.org/abs/2512.07686}, 
}

@article {Hanus2002,
    AUTHOR = {P. Hanus and R.D. Mauldin and M. Urba\'{n}ski},
     TITLE = {Thermodynamic formalism and multifractal analysis
of conformal infinite iterated function systems},
   JOURNAL = { Acta Math. Hungar. },
    VOLUME = {96},
      YEAR = {2002},
     PAGES = {27--98},
}

@article {Urbanski,
    AUTHOR = {P. Hanus and R.D. Mauldin and M. Urba\'{n}ski},
     TITLE = {{Diophantine analysis of conformal iterated function systems}},
   JOURNAL = { Monatsh. Math.},
    VOLUME = {137},
      YEAR = {2002},
     PAGES = {325-340},
}

@article {Tseng2009,
    AUTHOR = {Tseng, J.},
     TITLE = {Schmidt games and {M}arkov partitions},
   JOURNAL = {Nonlinearity},
  FJOURNAL = {Nonlinearity},
    VOLUME = {22},
      YEAR = {2009},
    NUMBER = {3},
     PAGES = {525--543},
      ISSN = {0951-7715,1361-6544},
   MRCLASS = {37D20 (28A78 37A05 37B10)},
MRREVIEWER = {Radhakrishnan\ Nair},
       DOI = {10.1088/0951-7715/22/3/001},
       URL = {https://doi.org/10.1088/0951-7715/22/3/001},
}

@article {Kleinbock2010,
    AUTHOR = {Kleinbock, D. and Weiss, B.},
     TITLE = {Modified {S}chmidt games and {D}iophantine approximation with
              weights},
   JOURNAL = {Adv. Math.},
  FJOURNAL = {Advances in Mathematics},
    VOLUME = {223},
      YEAR = {2010},
    NUMBER = {4},
     PAGES = {1276--1298},
      ISSN = {0001-8708,1090-2082},
   MRCLASS = {11J83 (11J13)},
MRREVIEWER = {Simon\ Kristensen},
       DOI = {10.1016/j.aim.2009.09.018},
       URL = {https://doi.org/10.1016/j.aim.2009.09.018},
}

@article{Kleinbock2005,
    AUTHOR = {Kleinbock, D. and Weiss, B.},
     TITLE = {Badly approximable vectors on fractals},
   JOURNAL = {Israel J. Math.},
  FJOURNAL = {Israel Journal of Mathematics},
    VOLUME = {149},
      YEAR = {2005},
     PAGES = {137--170},
      ISSN = {0021-2172,1565-8511},
   MRCLASS = {11J83 (11K55 28A80 37C45)},
MRREVIEWER = {Radhakrishnan\ Nair},
       DOI = {10.1007/BF02772538},
       URL = {https://doi.org/10.1007/BF02772538},
}

@misc{HussainSimmons2026,
  author       = {Hussain, Mumtaz and Simmons, David},
  title        = {Well and badly approximable sets, and rapid winning},
  year         = {2026},
  howpublished = {arXiv preprint arXiv:2608.25349v2},
  url          = {https://arxiv.org/abs/2608.25349v2}
}

@article {persson2024,
    AUTHOR = {Z.-N. Hu and T. Persson},
     TITLE = {Hausdorff dimension of recurrence sets},
   JOURNAL = {Nonlinearity},
  FJOURNAL = {Nonlinearity},
    VOLUME = {37},
      YEAR = {2024},
    NUMBER = {5},
     PAGES = {Paper No. 055010, 22},
      ISSN = {0951-7715,1361-6544},
   MRCLASS = {37C45 (28A80 37D20)},
MRREVIEWER = {Esa\ J\"arvenp\"a\"a},
       DOI = {10.1088/1361-6544/ad3597},
       URL = {https://doi.org/10.1088/1361-6544/ad3597},
}

\end{document}